%% file: main_arxiv.tex
\documentclass[smallcondensed]{svjour3}     
\smartqed  
\usepackage{graphicx}

\usepackage{xcolor}
\usepackage{comment}
\usepackage{todonotes}

\usepackage{amsmath}
\usepackage{amsfonts}
\usepackage{amssymb}
\usepackage{subfig}
\usepackage[noadjust]{cite}

\usepackage{algorithm}
\usepackage{algpseudocode}
\algnewcommand{\LIf}[1]{\State\algorithmicif\ {\footnotesize #1}\ \algorithmicthen}
\algnewcommand{\EndLIf}{\unskip\ \algorithmicend\ \algorithmicif}

\usepackage{tikz}
\usetikzlibrary{shapes,arrows}
\tikzstyle{line} = [draw, -latex']

\newcommand{\blue}[1]{{\color{black}#1}}

\newcommand{\myNorm}[2][]{||#2||_{#1}}

\newcommand{\mySP}[3][]{(#2,#3)_{#1}}

\newcommand{\myProj}[1]{\Pi_{#1}}

\newcommand{\Hkp}[1][k,p]{H^{#1}(\Omega)}

\newcommand{\myU}{\mathcal{U}}
\newcommand{\myY}{\mathcal{Y}}
\newcommand{\myT}{\mathcal{T}}

\newcommand{\myHeins}{\mathcal{H}}
\newcommand{\myHzwei}{\myY}

\newcommand{\mySPU}[2]{\mySP[\myU]{#1}{#2}}
\newcommand{\mySPY}[2]{\mySP[\myY]{#1}{#2}}
\newcommand{\myNormU}[1]{\myNorm[\myU]{#1}}
\newcommand{\myNormY}[1]{\myNorm[\myY]{#1}}

\newcommand{\mya}[1][\mu]{a_{#1}}
\newcommand{\myb}[1][\mu]{b_{#1}}
\newcommand{\myB}[1][\mu]{B_{#1}}
\newcommand{\myG}[1][\mu]{G_{#1}}

\newcommand{\myalpha}[1][\mu]{\alpha_a(#1)}
\newcommand{\myalphaLB}[1][\mu]{\alpha_a^{\rm{LB}}(#1)}
\newcommand{\mybeta}[1][\mu]{\beta_B(#1)}
\newcommand{\mybetaLB}{\beta_B^{\rm{LB}}}
\newcommand{\mydelta}[1][\mu,\lambda]{\alpha_A^0(#1)}
\newcommand{\mydeltaLB}{\alpha^{\rm{LB}}_{A}}
\newcommand{\mybUB}[1][\mu]{\gamma_{b}^{\rm{UB}}(#1)}
\newcommand{\mykappa}[1][\mu]{\beta_{\myT}(#1)}

\newcommand{\myybk}[1][\mu]{y_{\rm{bk},#1}}
\newcommand{\myfbk}{f_{\rm{bk},\mu}}
\newcommand{\myytrue}{y_{\rm{true}}}
\newcommand{\mydata}{y_{\rm{d}}}

\newcommand{\myu}[1][\mu]{u^*_{#1}}
\newcommand{\myy}[1][\mu]{y^*_{#1}}
\newcommand{\myd}[1][\mu]{d^*_{#1}}
\newcommand{\myw}[1][\mu]{p^*_{#1}}

\newcommand{\myheins}[1][\mu]{h^*_{#1}}
\newcommand{\myhzwei}[1][\mu]{\myw[#1]}

\newcommand{\myURB}{\mathcal{U}_{\rm{R}}}
\newcommand{\myYRB}{\mathcal{Y}_{\rm{R}}}
\newcommand{\myHeinsRB}{\mathcal{H}_{\rm{R}}}

\newcommand{\myuRB}[1][\mu]{u^*_{\rm{R,}#1}}
\newcommand{\myyRB}[1][\mu]{y^*_{\rm{R,}#1}}
\newcommand{\mydRB}[1][\mu]{d^*_{\rm{R,}#1}}
\newcommand{\mywRB}[1][\mu]{p^*_{\rm{R,}#1}}

\newcommand{\myheinsRB}[1][\mu]{h_{\rm{R,}#1}^*}
\newcommand{\myhzweiRB}[1][\mu]{\mywRB[#1]}
\newcommand{\myybkRB}[1][\mu]{y_{\rm{bk,R},#1}}

\newcommand{\myeu}{e_{u}}
\newcommand{\myey}{e_{y}}
\newcommand{\myed}{e_{d}}
\newcommand{\myew}{e_{p}}

\newcommand{\myNreins}{\myNorm[\myU']{r_{u}}}
\newcommand{\myNrzwei}{\myNorm[\myY']{r_{p}}}
\newcommand{\myNrdrei}{\myNorm[\myY']{r_{y}}}
\newcommand{\myru}{r_{u}}
\newcommand{\myry}{r_{y}}
\newcommand{\myrp}{r_{p}}

\newcommand{\myM}{\mathcal{M}}
\newcommand{\myN}{\mathcal{N}}
\newcommand{\myvecu}{\mathbf{u}^{*}}
\newcommand{\myvecy}{\mathbf{y}^{*}}

\newcommand{\myvecw}{\mathbf{p}^{*}}

\newcommand{\myvecfbk}{\mathbf{f}_{\rm{bk}}}
\newcommand{\myvecdata}{\mathbf{y}_{\rm{d}}}

\newcommand{\myvecmeasured}{\mathbf{m}_{\rm{d}}}
\newcommand{\myvecuRB}[1][]{\mathbf{u}^{*#1}_{\rm{R}}}
\newcommand{\myvecyRB}[1][]{\mathbf{y}^{*#1}_{\rm{R}}}

\newcommand{\myvecwRB}[1][]{\mathbf{p}^{*#1}_{\rm{R}}}

\newcommand{\myvecdataRB}{\mathbf{y}_{\rm{d,R}}}
\newcommand{\myvecfbkRB}{\mathbf{f}_{\rm{bk,R}}}

\newcommand{\myDu}[1][\mu]{\Delta_u(#1)}
\newcommand{\myDy}[1][\mu]{\Delta_y(#1)}
\newcommand{\myDd}[1][\mu]{\Delta_d(#1)}
\newcommand{\myDw}[1][\mu]{\Delta_p(#1)}

\newcommand{\myGdir}{\Gamma_{\rm{D}}}
\newcommand{\myGneu}{\Gamma_{\rm{N}}}
\newcommand{\myGin}{\Gamma_{\rm{in}}}
\newcommand{\myYcont}{\mathcal{Y}_{\rm{e}}}
\newcommand{\myUcont}{\mathcal{U}_{\rm{e}}}

\newcommand{\myutrue}{u_{\rm{true}}}
\newcommand{\mymutrue}{\mu _{\rm{true}}}

\newcommand{\mymuzwei}[1][\lambda]{\mu^{*,#1}_{2}}
\newcommand{\mymudrei}[1][\lambda]{\mu^{*,#1}_{3}}
\newcommand{\mymui}[1][\lambda]{\mu^{*,#1}_{i}}

\newcommand{\mymuiRB}[1][\lambda]{\mu^{*,#1}_{i\rm{,R}}}
\newcommand{\myustart}{u_{\rm{start}}}

\newcommand{\myJeins}[1][\lambda]{J_1^{#1}}
\newcommand{\myJzwei}[1][\lambda]{J_2^{#1}}
\newcommand{\myJdrei}[1][\lambda]{J_3^{#1}}
\newcommand{\myJi}[1][\lambda]{J_i^{#1}}
\newcommand{\myJeinsRB}[1][\lambda]{J_{1\rm{,R}}^{#1}}
\newcommand{\myJzweiRB}[1][\lambda]{J_{2\rm{,R}}^{#1}}
\newcommand{\myJdreiRB}[1][\lambda]{J_{3\rm{,R}}^{#1}}
\newcommand{\myJiRB}[1][\lambda]{J_{i\rm{,R}}^{#1}}

\newcommand{\myyPBDW}{y_{\infty,\mu}}
\newcommand{\myyPBDWhat}{\hat{y}_{\infty,\mu}}
\newcommand{\myetaPBDW}{d_{\infty,\mu}}
\newcommand{\myuPBDW}{u_{\infty,\mu}}

\newcommand{\myYmu}[1][\mu]{\mathcal{Y}_{#1}}
\newcommand{\myYmuRB}[1][\mu]{\mathcal{Y}_{\rm{R,}#1}}

\newcommand{\myetainf}{\underline{\eta}(\mu)}
\newcommand{\myetasup}{\overline{\eta}(\mu)}
\newcommand{\myetainfRB}{\underline{\eta}_{\rm{R}}(\mu)}

\newcommand{\mykappaRBmu}[1][\mu]{\beta_{\myT,\rm{R}}(#1)}

\newcommand{\mygammaa}[1][\mu]{\gamma_{\rm{a}}(#1)}
\newcommand{\mygammab}[1][\mu]{\gamma_{\rm{b}}(#1)}
\newcommand{\mygammaA}{\gamma_{\rm{A}}}
\newcommand{\mygammaB}[1][\mu]{\gamma_{\rm{B}}(#1)}

\usepackage{todonotes}

\begin{document}

\title{
3D-VAR for Parametrized Partial Differential Equations: A Certified Reduced Basis Approach
\thanks{This work was supported by the Excellence Initiative of the German federal and state governments and the German Research Foundation through Grant GSC 111.}
}

\titlerunning{A certified reduced basis approach for 3D-\blue{VAR}}        

\author{	Nicole Aretz-Nellesen         	\and
        			 Martin A. Grepl			\and \\
        			 Karen Veroy
}

\institute{
Nicole Aretz-Nellesen \at
Aachen Institute for Advanced Study in Computational Engineering Science (AICES),
RWTH Aachen University, Schinkelstra{\ss}e 2, 52062 Aachen, Germany \\
\email{nellesen@aices.rwth-aachen.de}
\and
Martin A. Grepl \at
Numerical Mathematics (IGPM), RWTH Aachen University, Templergraben 55, 52056 Aachen, Germany \\
\email{grepl@igpm.rwth-aachen.de}
\and
Karen Veroy \at
Aachen Institute for Advanced Study in Computational Engineering Science (AICES)
and Faculty of Civil Engineering, RWTH Aachen University, Schinkelstraße 2, 52062 Aachen \\
\email{veroy@aices.rwth-aachen.de}
}

\date{Received: date / Accepted: date}
%
%
\maketitle
\begin{abstract}
In this paper, we propose a reduced order approach for 3D variational data assimilation governed by parametrized partial differential equations. In contrast to the classical 3D-\blue{VAR} formulation that penalizes the measurement error directly, we present a modified formulation that penalizes the experimentally-observable misfit in the measurement space. Furthermore, we include a model correction term that allows to obtain an improved state estimate. We begin by discussing the influence of the measurement space on the amplification of noise and prove a necessary and sufficient condition for the identification of a ``good" measurement space. We then propose a certified reduced basis (RB) method for the estimation of the model correction, the state prediction, the adjoint solution and the observable misfit with respect to the true state for real-time and many-query applications. {\it A posteriori} bounds are proposed for the error in each of these approximations. Finally, we introduce different approaches for the generation of the reduced basis spaces and the stability-based selection of measurement functionals. The 3D-VAR method and the associated certified reduced basis approximation are tested in a parameter and state estimation problem for a steady-state thermal conduction problem with unknown parameters and unknown Neumann boundary conditions.

\end{abstract}


\section{Introduction}
\label{sec:introduction}

In numerical simulations, mathematical models --- such as ordinary or partial differential equations (PDEs) --- are widely used to predict the state or behavior of a physical system. The goal of variational data assimilation is to improve state predictions through the incorporation of measurement data, e.g., experimental observations, into the mathematical model. \blue{Variational data assimilation is prevalent in meteorology \cite{DT1986,Lorenc1986,Lorenc1981} and oceanography \cite{Bennett,Bennett_2002}, for example in weather forecasting and ocean circulation modeling.}  Prominent examples include the 3D- and 4D-\blue{VAR} methods, which weigh deviations from a prior best-knowledge (initial) state against differences with respect to measurement data; see the recent texts~\cite{DataAssimilation,Reich} and references therein for a discussion of variational data assimilation.

Whereas 4D-\blue{VAR} considers dynamical systems (i.e. three space dimensions plus time), and usually aims to estimate the (unknown) initial condition of the system, 3D-\blue{VAR} considers the stationary case. In this paper, we propose a certified reduced order approach for a modified 3D-VAR method for parametrized, coercive PDEs. Compared to the classical 3D-\blue{VAR} method, our modified formulation penalizes the experimentally-observable misfit in the measurement space instead of the difference in the measurements. Furthermore, we account for an imperfect model by introducing a model bias in the formulation, similar to the weak-constraint 4D-\blue{VAR} approach~\cite{Tremolet2006}.

The proposed method makes data-informed modifications to a best-knowledge background model to generate a compromise between the original model and the observed measurements.
The method thereby accounts for the physical integrity of the state prediction and can be used to estimate unknown model properties and parameters along with the state. The 3D-\blue{VAR} problem --- and variational data assimilation in general --- is usually cast as an optimization problem and has very close connections to optimal control theory~\cite{VH2006}.
In addition, the optimality condition has a saddle-point structure which can be analyzed using standard functional analysis arguments. Based on these observations we can identify necessary and sufficient properties of the measurement space that \textit{increase} the stability of the 3D-VAR formulation and provide a practical procedure for the generation of a measurement space meeting these conditions.

This paper builds upon the model-data weak approach presented in~\cite{YANO2013937}, the parametrized-background data-weak (PBDW) approach to variational data assimilation introduced in~\cite{PBDW,PBDW2}, and the certified reduced-order approach for 4D-\blue{VAR} in~\cite{4DVAR}.
The model-data weak formulation in \cite{YANO2013937} finds a state estimate by minimizing the distance between state and observed state while penalizing model corrections.
This method indirectly accepts any kind of model modification in the dual space and can thereby account for unpredicted behaviour.  In the parametrized-background data-weak (PBDW) approach to variational data assimilation introduced in \cite{PBDW,PBDW2}, a state estimate is obtained by projecting an observation from the measurement space onto a space featuring model properties, such as an RB space that approximates the solution manifold of a parametrized PDE.
The PBDW framework has recently been extended with adaptive \cite{Taddei17,2017arXiv171209594M} and localization \cite{Taddei_localization} strategies, and has been recast in a multispace setting \cite{doi:10.1137/15M1025384}.
In addition to the discussions in the PBDW literature, the optimal selection of measurements with greedy orthogonal matching pursuit (OMP) algorithms in state estimation has been analyzed recently in~\cite{Olga}.
We build upon this work in the construction of our measurement space.

The connection and distinctions between 3D-\blue{VAR} and the PBDW formulation has already been discussed in~\cite{PBDW,Taddei17}. We also note that the PBDW approach is related to the generalized empirical interpolation method~\cite{Maday2013,Maday_gEIM_2015} and gappy-POD~\cite{WILLCOX2006208}, see~\cite{PBDW,Taddei17} for a thorough discussion. We will show that --- under certain assumptions on the spaces {\it and} in the limit of the regularization parameter going to infinity --- our modified 3D-\blue{VAR} formulation is equivalent to the PBDW formulation. However, there are also decisive differences between the two formulations.

After a brief discussion of the mathematical background in Section~\ref{sec:preliminaries}, we present the following main contributions:
\begin{itemize}
\item We introduce a data-weak reformulation of the classical 3D-VAR method in Section \ref{sec:truth}. We present a detailed analysis of the method's stability with respect to the properties of the original best-knowledge model, the model bias, and the choice of the measurement space. We show that our method is stable in the sense of Hadamard independently of the choice of measurement functionals and uniformly over the parameter domain. Furthermore, we identify a necessary and sufficient property of the measurement space that further increases the stability and restricts the amplification of measurement noise when the 3D-VAR method emphasizes closeness to the data over the best-knowledge model.
\item In Section \ref{sec:RB}, we present the RB method for the 3D-\blue{VAR} problem and develop online computationally efficient approximations and {\it a posteriori} error bounds for the state estimate, the adjoint solution, the model correction, and the misfit. This significant reduction in computational complexity makes the 3D-\blue{VAR} method feasible for repeated computations with measurement noise and varying parameters. Furthermore, the
RB approximation aids in the efficient selection of the measurements.
\item In Section \ref{sec:spaces}, we integrate the theoretical results from both the stability and error analyses to propose an iterative selection of the measurement functionals. More specifically, we propose an algorithm that employs
the OMP measurement selection from \cite{Olga}
in a greedy manner over the parameter domain to generate the measurement space. The space can be tailored specifically to the stable estimation of parameters and model properties of the 3D-VAR method from measurement data. We also discuss a construction of the reduced basis spaces which does not require the measurements to be known {\it a priori}.
\end{itemize}
In Section \ref{sec:experiments} we present numerical results for a steady state thermal conduction problem with uncertain parameters and unknown Neumann boundary condition, and investigate the influence of measurement noise upon the estimation of the uncertain parameters and boundary condition.


\section{Preliminaries}
\label{sec:preliminaries}

In this section, we specify the mathematical framework within which we cast the 3D-VAR method in the following section. We introduce the required spaces and forms, and list all our assumptions upon them.

We start with the state space $\myY$, which we take to be a Hilbert space with inner product $\mySPY{\cdot}{\cdot}$ and induced norm $\myNormY{\cdot}$. Likewise, we consider a Hilbert space $\myU$ for the regularizing model correction, with norm $\myNormU{\cdot}$ induced by the inner product $\mySPU{\cdot}{\cdot}$. In the numerical implementation, these spaces are typically finite-dimensional subspaces --- in the RB literature usually referred to as ``truth'' spaces --- of Sobolev- or $L^2$-spaces, e.g. finite element spaces. For notational convenience, we omit the explicit distinction between the infinite dimensional and finite dimensional setting in the following since the results presented hold for both cases (with the appropriate definitions). The RB approximation in section \ref{sec:RB} utilizes closed subspaces $\myYRB \subset \myY$ and $\myURB \subset \myU$ that represent the most dominant model dynamics.

For the incorporation of measurement data, we additionally consider a nontrivial, finite-dimensional subspace $\myT \subset \myY$ as measurement space and let the operator $\myProj{\myT} : \myY \rightarrow \myT$ denote the orthogonal projection onto this space. \blue{We assume $\dim \myT < \infty$, but note that most of our analysis holds for a generic closed subspace $\myT \subset \myY$.}
\begin{remark}\label{rmk:T}
\blue{We} briefly \blue{explain how the space $\myT$} can be linked to physical measurements:
Suppose we are given $L < \infty$ linearly independent measurement functionals $g_l \in \myY'$, $l \in \{1,...,L\}$.
We can then choose the hierarchical space $\myT$ as the span of the Riesz representations of the measurement functions, see \cite{BENNETT1985129,PBDW}, via
\begin{align} \label{eq:def:T}
\myT ~=~\text{span}\{~\tau_l \in \myT:~1 \le l \le L \text{ and } \mySPY{\tau_l}{\tau} = g_l(\tau)~\forall~\tau \in \myT~\}.
\end{align}
For any state $y \in \myY$, the projection $\myProj{\myT}y$ is then the only state in $\myT$ that yields the same measurements as $y$, i.e. $g_l(\myProj{\myT}y) = g_l(y) \in \mathbb{R}$ for $l=1,...,L$. In this context, we can consider $\myProj{\myT}y$ to be the experimentally observed part of $y$. Due to the linear independence of the measurement functionals, any set $\mathbf{m} = (m_l)_{l=1}^L \in \mathbb{R}^L$ is obtained as the measurement data of exactly one state \blue{$\mydata(\mathbf{m})$} in $\myT$\blue{, and -- due to the projection theorem -- of all states in $\mydata(\mathbf{m}) + \myT^{\perp}$. More specifically, t}wo states in $\myY$ yield the same measurements if and only if their difference lies in $\myT^{\perp}$. In practice, the number $L$ of measurements should be kept small to limit experimental expenses.
\end{remark}

We let $\mathcal{C} \subset \mathbb{R}^d$ be a compact set of all admissible parameters. For this set, we introduce three (possibly) parameter-dependent forms, namely $\myfbk \in \myY'$ and the bilinear forms $\mya : \myY \times \myY \rightarrow \mathbb{R}$ and $\myb : \myU \times \myY \rightarrow \mathbb{R}$.
\blue{Throughout this paper, the index $\rm{bk}$ stands for ``best-knowledge," whereas the index $\mu$ signifies the dependence on a parameter $\mu \in \mathcal{C}$.}
The first two forms represent the best-knowledge model dynamics
\begin{align} \label{eq:bk}
\text{find } \myybk \in \myY \text{ such that } \mya(\myybk,\psi) &= \myfbk(\psi) \qquad \forall~\psi \in \myY,
\end{align}
whereas for $u \in \myU$, $\myb(u,\cdot) \in \myY'$ denotes the induced model modification.

We make the assumptions that $\mya$ is uniformly coercive over $\mathcal{C}$\footnote{Throughout this paper, we adopt the notational convention that infima and suprema exclude elements with norm 0.},
\begin{align}\label{eq:def:alpha}
\exists~\underline{\alpha}_a > 0 \text{ s.t. }
\myalpha :=  \inf _{y \in \myY} \frac{\mya(y,y)}{\myNormY{y}^2} ~\ge~ \underline{\alpha}_a \quad \forall \mu \in \mathcal{C},
\end{align}
and that both bilinear forms $\mya$ and $\myb$ are uniformly bounded, i.e.,
\begin{equation}\label{eq:uniformBound:ab}
\begin{aligned}
&\exists~\overline{\gamma}_a > 0 \text{ s.t. } 0 ~<~
\mygammaa := \sup_{y \in \myY} \sup_{z \in \myY}  \frac{\mya(y,z)}{\myNormY{y} \myNormY{z}} ~\le~ \overline{\gamma}_a ~<~ \infty \quad \forall \mu \in \mathcal{C},\\
&\exists~ \overline{\gamma}_b > 0 \text{ s.t. } 0 ~<~ \mygammab := \sup_{u \in \myU}\sup_{y \in \myY}   \frac{\myb(u,y)}{\myNormU{u}\myNormY{y}} ~\le~ \overline{\gamma}_b ~<~ \infty \quad \forall \mu \in \mathcal{C}.
\end{aligned}
\end{equation}
The conditions \eqref{eq:def:alpha} and \eqref{eq:uniformBound:ab} guarantee that the best-knowledge model \eqref{eq:bk} and the 3D-VAR formulation introduced in the next section are uniformly stable and that the error of the RB approximation is quasi-optimal with uniformly bounded constants over $\mathcal{C}$. In anticipation of the a posteriori error estimation procedure, we additionally presuppose that we can compute a lower bound  $\myalphaLB \le \myalpha$ and an upper bound  $\mybUB \ge \mygammab$ at reasonably low cost for all $\mu \in \mathcal{C}$, e.g., through the min-$\theta$-approach or the successive constraint method~\cite{SCM,RB_Patera_2007}.

For the efficient computation of the RB solution and a posteriori error bounds in an offline-online procedure, we make the assumption that the $\mu$-dependent forms $\mya$, $\myb$ and $\myfbk$ are affine in functions of the parameter, i.e.,
\begin{equation} \label{eq:affine}
\begin{aligned}
\mya~=~ \textstyle\sum\limits _{\vartheta = 1} ^{\Theta_a} \theta^{\vartheta}_a(\mu) a^{\vartheta}  \qquad
\myb~=~ \sum\limits _{\vartheta = 1} ^{\Theta_b} \theta^{\vartheta}_b(\mu) b^{\vartheta} \qquad
\myfbk~=~ \sum\limits_{\vartheta = 1} ^{\Theta_f} \theta^{\vartheta}_f(\mu) f^{\vartheta}_{\rm{bk}},
\end{aligned}
\end{equation}
where the coefficient functions $\theta_a$, $\theta_b$, $\theta_f : \mathcal{C} \rightarrow \mathbb{R}$ are continuous in the parameter $\mu$, and the bilinear forms $a^{\vartheta}$, $b^{\vartheta} $ as well as the linear forms $f^{\vartheta}_{\rm{bk}}$ are parameter-independent.
In the nonaffine case, (generalized) empirical interpolation techniques may be used to construct an affine approximation \cite{Barrault,Grepl_nonaffine,Maday_gEIM_2015,Maday2013}.


\section{3D-VAR Formulation}
\label{sec:truth}
In the following, we introduce the 3D-VAR formulation for a coercive, parameter-dependent PDE.
The method aims at finding a weighted compromise between a best-knowledge model and  measurement data by making a data-informed perturbation of the model. After a reformulation as a saddle-point problem, a stability analysis shows how the choice of measurement functionals influences the amplification of errors in the measurements and the approximation of the best-knowledge source term. This leads to a practical criterion for the selection of suitable measurement functionals.

\subsection{Problem Statement}
\label{sec:intro3DVAR}

The main goal in data assimilation is to estimate the (unknown) state $\myytrue \in \myY$ of a physical system. To this end, we suppose that a best-knowledge mathematical model exists in terms of the underlying elliptic PDE \eqref{eq:bk} whose unique solution $\myybk$ provides a first approximation of $\myytrue$. However, any model can only provide an approximation to the underlying physics due to simplifications in the derivation of the governing PDE, and boundary conditions, or due to approximations and uncertainties of the system geometry or loading, $\myfbk$. Since $\myybk \neq \myytrue$ in general, we thus introduce a perturbation of the best-knowledge problem \eqref{eq:bk} in order to allow for a better approximation of the state $\myytrue$. Here, we specifically consider model corrections of the form $\myb(u,\cdot) \in \myY '$ with $u \in \myU$, leading to the modified problem:
\begin{equation}\label{eq:forward}
\begin{aligned}
&\text{For a given \blue{model modification }}u_0 \in \myU \text{ find } y_{\mu} = y_{\mu}(u_0) \in \myY \text{ s.t. }\\
&\mya(y_{\mu},\psi) = \myfbk (\psi) + \myb(u_0,\psi) \quad \forall~\psi \in \myY.
\end{aligned}
\end{equation}
Within this framework, we now aim to construct an improved approximation of $\myytrue$ by finding a ``good" model correction $u \in \myU$.

To do this, we incorporate knowledge \blue{of the true state $\myytrue$ in form of an approximation $\mydata \in \myT$ to $\myProj{\myT}\myytrue$.
 In the notation of Remark \ref{rmk:T}, this approximation might be $\mydata = \mydata(\myvecmeasured) \in \myT$, resulting from measurement data $\myvecmeasured \in \mathbb{R}^L$, $(\myvecmeasured)_l = g_l(\myytrue) + \varepsilon_l$ of the true state with noise $\varepsilon_l \in \mathbb{R}$.
To approximate $\myytrue \in \myProj{\myT}\myytrue + \myT^{\perp} \approx \mydata + \myT^{\perp}$, the}
model correction $u \in \myU$ should thus be chosen such that the solution $y_{\mu}(u)$ of the modified model \eqref{eq:forward} for $u_0 = u$ is close to $\blue{\mydata + \myT^{\perp}}$.
\blue{By the projection theorem, this corresponds to choosing $u$ such that} the experimentally-observable misfit $d_{\mu}(u) = \mydata-\myProj{\myT}y_{\mu}(u)$ between \blue{$y_{\mu}(u)$ and} the data state $\mydata$, measured in the $\myY$-norm, \blue{is} small.
At the same time, however, we want the model correction, measured in $\myNormU{u}$, to be small such that the state \blue{$y_{\mu}(u)$} stays close to our best-knowledge model \eqref{eq:bk}. Note that \blue{$y_{\mu}(0) = \myybk$} for $u_0 = 0 \in \myU$ in~\eqref{eq:forward}. The 3D-VAR formulation thus takes the form
\begin{equation}\label{prob:truth}
\begin{aligned}
& \min _{(u,y,d) \in \, \myU \times \myY \times \myT} \frac{1}{2} \myNormU{u}^2 + \frac{\lambda}{2} \myNormY{d}^2 ~\text{ s.t. }\\
& \left.
\begin{array}{rll}
\mya(y,\psi) &= \myfbk (\psi) + \myb(u,\psi) &\quad  \forall~\psi \in \myY,\\
\mySP{y+d}{\tau} &= \mySP{\mydata}{\tau} &\quad \forall ~\tau \in \myT,
\end{array}
\right.
\end{aligned}
\end{equation}
where the regularization parameter $\lambda > 0$ quantitatively expresses the trust in the original model \eqref{eq:bk} over the validity of the measurements: a small factor $\lambda$ prioritizes proximity to the original model, whereas a large $\lambda$ favours closeness to the data state $\mydata$. By specifically considering model corrections and the distance in the measurement space $\myT$ (instead of the difference in measurements), we follow the model-data weak approach to variational data assimilation introduced in \cite{YANO2013937}.

Denoting the optimal solution by $(\myu,\myy,\myd) \in \myU \times \myY \times \myT$, then $\myy = y_{\mu}(\myu)$ is the solution of the perturbed model \eqref{eq:forward} with model correction $u_0 = \myu$, and $\myd = \mydata-\myProj{\myT}\myy$ is the experimentally-observable misfit between the data state $\mydata$ and the 3D-VAR solution state $\myy$.
The minimization over $\myNormU{u}$ in the cost function enforces that the model correction $\myu$ contains only components that actively decrease the misfit, i.e. $\myu$ is perpendicular to the closed subspace
\begin{align} \label{eq:def:U0}
\myU_0(\mu) := \{~ u \in \myU:~ \myProj{\myT}y_{\mu} = 0 \text{ where } \mya(y_{\mu},\psi) ~=~ \myb(u,\psi)~\forall \psi \in \myY~\}.
\end{align}
If the measurements can be reproduced in the modified model \eqref{eq:forward}, that is if $\mydata = \myProj{\myT}y_{\mu}(\myutrue)$ for a solution $y_{\mu}(\myutrue)$ of \eqref{eq:forward} with $u_0 = \myutrue \in \myU$, then $\frac{1}{2}\myNormU{\myutrue}^2$ is a natural, $\lambda$-independent upper bound for the cost function in \eqref{prob:truth}.
Therefore, $\myNormY{\myd} \in \mathcal{O}(\lambda^{-1/2})$, and $\myNormU{\myu} \le \myNormU{\myutrue}$ is bounded independently of $\lambda$.

\begin{remark}
Formally, the 3D-VAR problem \eqref{prob:truth} is a Tikhonov regularisation
\begin{align}\label{eq:Tikhonov}
\min _{u \in \myU}\frac{1}{\lambda} \myNormU{u}^2 +  \myNormY{Q_{\mu}u-\mydata +\myProj{\myT}\myybk} ^2
\end{align}
with regularisation parameter $1/\lambda$ and bounded linear operator $Q_{\mu} : \mathcal{U} \rightarrow \mathcal{T}$, $Q_{\mu}u_0 := \myProj{\myT}y_{\mu}(u_0)$, where $y_{\mu}(u_0) \in \myY$ is the unique solution of \eqref{eq:forward} with source term $\myfbk = 0$. With standard theory for Tikhonov regularisations (see e.g.~\cite{Tikhonov_1995})  there exists a unique solution $\myu = \myu(\lambda,\mydata) \in \myU$ to \eqref{eq:Tikhonov} and hence to the 3D-VAR problem \eqref{prob:truth} for all data states $\mydata \in \myT$ and all $\lambda > 0$. In particular, the solution depends continuously on the data with $\myNormU{\myu(\lambda,\mydata)} \le \sqrt{\lambda} \myNormY{\mydata - \myProj{\myT}\myybk}$. Thus, the 3D-VAR problem \eqref{prob:truth} is well-posed.
\end{remark}

\begin{remark} \label{rmk:convergence}
\blue{Since $\text{range}(Q_{\mu}) \subset \myT$ and $\dim \myT < \infty$, we can exploit} the connection between Tikhonov regularisations and generalized inverses~\cite{Groetsch_1977} \blue{to} obtain that for $\lambda \rightarrow \infty$ the solution $\myu(\lambda,\mydata)$ converges to the generalized inverse $Q_{\mu}^+ (\mydata- \myProj{\myT}\myybk)$\blue{, which is the unique solution to
\begin{align*}
\min _{u \in \myU} \myNormU{u} \quad \text{s.t. } \quad u \in \text{arg} \min _{v \in \myU} \myNormY{Q_{\mu}v-\mydata+ \myProj{\myT}\myybk}.
\end{align*}
If the true state $\myytrue$ can be completely described by \eqref{eq:forward} for a model modification $u_0 = \myutrue$, and if in addition $\mydata = \myProj{\myT}\myytrue$ is an unbiased observation (e.g. from noise-free measurements), then $\myNormY{Q_{\mu}u-\mydata+ \myProj{\myT}\myybk} = 0$ if and only if $u \in \myutrue + \myU_0(\mu)$.
The 3D-VAR model correction $\myu$ then converges to $\myProj{\myU_0 ^{\perp}(\mu)} \myutrue$.
}
\end{remark}

%
\subsection{Saddle-Point Formulation}
\label{sec:tSaddlePoint}
We next recast the 3D-VAR minimization \eqref{prob:truth} as a saddle-point problem. Although we already know that the minimization is well-posed as a Tikhonov regularisation for any fixed parameter $\mu$, the saddle-point structure has several benefits:
First, the analytic properties known for saddle-point problems --- such as Brezzi's theorem ---  allow us to quantify the structural influences of the original best-knowledge model \eqref{eq:bk}, the model perturbation $\myb$, and the measurement space $\myT$ on the amplification of errors in $\mydata$ and $\myfbk$. Second, the analysis shows that the 3D-VAR problem is uniformly well-posed over the parameter domain, enabling the use of the method in parameter estimation. Third, in the discretized setting we obtain a linear system with saddle-point structure for which various numerical methods exist.

We note that for the solution $(\myu,\myy,\myd) \in \myU \times \myY \times \myT$ of \eqref{prob:truth}, the experimentally-observable mistfit, $\myd = \mydata-\myProj{\myT}\myy$, between the 3D-VAR solution $\myy$ and the observed state $\mydata$ can be be computed {\it a posteriori} from $\myu$ and $\myy$. We therefore omit the explicit dependence on $\myd$ in the saddle point formulation which results not only in a smaller system but also improved stability constants. For an approach which explicitly includes $\myd$, we refer to
\cite{Masterarbeit_Nicole}.

We next define the Hilbert space $\myHeins := \myU \times \myY$ with induced inner product $\mySP[\myHeins]{(u,y)}{(\phi,\psi)} := \mySPU{u}{\phi} + \mySPY{y}{\psi}$ and induced norm $\myNorm[\myHeins]{\cdot} := \sqrt{\mySP[\myHeins]{\cdot}{\cdot}}$.
Furthermore, we define the bilinear and linear forms
\begin{equation}\label{eq:defABG}
\begin{aligned}
A : \myHeins \times \myHeins \rightarrow \mathbb{R}: && A((u,y),(\phi,\psi)) & := \mySPU{u}{\phi} + \lambda \mySPY{\myProj{\myT}y}{\myProj{\myT}\psi} \\
\myB : \myHeins \times \myHzwei \rightarrow \mathbb{R}: && \myB((u,y),\psi) & := \mya(y,\psi) - \myb(u,\psi) \\
F : \myHeins \rightarrow \mathbb{R}: && F ((\phi,\psi)) &:= \lambda\mySPY{\mydata}{\psi} \\
\myG : \myHzwei \rightarrow \mathbb{R}: && \myG (\psi) &:= \myfbk(\psi).\\
\end{aligned}
\end{equation}
Note that $F$, $\myG$ are linear with $\myNorm[\myHeins']{F} = \lambda\myNormY{\mydata}$ and $\myNorm[\myHzwei']{\myG} = \myNorm[\myY']{\myfbk}$.
Furthermore, $A$ is symmetric and positive semi-definite, and both $A$ and $\myB$ are bilinear with bounded continuity constants
\begin{equation}\label{eq:def:gammaAB}
\begin{aligned}
\mygammaA &~:=~ \sup_{h \in \myHeins} \sup_{k \in \myHeins} \frac{A(h,k)}{\myNorm[\myHeins]{h} \myNorm[\myHeins]{k}} ~=~ \max\{1,\lambda\} \\
\mygammaB &~:=~ \sup_{h \in \myHeins} \sup_{y \in \myHzwei} \frac{\myB(h,y)}{\myNorm[\myHeins]{h} \myNormY{y}} ~\le~ \sqrt{\mygammaa^2 + \mygammab^2}. \\
\end{aligned}
\end{equation}
With the definitions in \eqref{eq:defABG}, the 3D-VAR minimization \eqref{prob:truth} is equivalent to
\begin{align*}
\min _{h \in \myHeins} A(h,h) - F(h) \text{ such that } \myB(h,\psi) = \myG(\psi) \text{ for all } \psi \in \myY.
\end{align*}
Employing a Lagrangian approach, we obtain the associated necessary, and in our setting sufficient, first-order optimality conditions~\cite{Hinze_book_optimizationPDE,Reyes_book}: Given $\mu \in \mathcal{C}$ find $(\myheins,\myhzwei)=((\myu,\myy),\myhzwei) \in \myHeins \times \myHzwei$ such that
\begin{equation}\label{prob:truth:sp}
\begin{aligned}
A(\myheins,k) + \myB(k,\myhzwei) &= F(k) && \forall~k \in \myHeins,  \\
\myB(\myheins,\psi) &= \myG(\psi) && \forall~\psi \in \myY.
\end{aligned}
\end{equation}

%
\subsection{Stability Analysis}
\label{sec:tStability}

The {\it a priori} stability of the 3D-Var formulation is closely linked to the choice of the measurement space $\myT$. The following stability analysis thus provides a criterion for choosing the measurement space $\myT$ in section \ref{sec:greedyOMP}. Furthermore, it allows us to link the properties of the modified model \eqref{eq:forward} and the choice of the measurement space $\myT$ to the amplification of noise in the 3D-Var solution. We start with the Ladyzhenskaya-Babu\v{s}ka-Brezzi (LBB) condition\blue{, for which we prove the lower bound  $\mybetaLB(\mu) := \myalpha$. In the sequel, the superscripts LB and UB shall refer to lower and upper bounds, respectively.}
\begin{theorem} \label{thm:tInfSup}
\begin{align*}
\mybeta ~:=~ \inf \limits_{y \in \myY  } \sup \limits_{h \in \myHeins  } \frac{\myB(h,y)}{\myNorm[\myHeins]{h}\myNormY{y}} ~\ge~ \mybetaLB(\mu) ~:=~\myalpha ~\ge~ \underline{\alpha}_a ~>~ 0.
\end{align*}
\end{theorem}

\begin{proof}
For any $y \in \myY\setminus \{0\}$, we obtain
\begin{align*}
\sup \limits_{h \in \myHeins  } \frac{\myB(h,y)}{\myNorm[\myHeins]{h}\myNormY{y}}
= \sup \limits_{(\phi,\psi) \in \myHeins  } \frac{\mya(\psi,y)-\myb(\phi,y)}{\myNorm[\myHeins]{(\phi,\psi)}\myNormY{y}}
\ge \frac{\mya(y,y)}{\myNormY{y}^2} \ge \myalpha,
\end{align*}
where we have chosen $(\phi,\psi) = (0,y)$ and inserted the coercivity condition \eqref{eq:def:alpha}.
\qed
\end{proof}
We now turn to verifying the coercivity of $A$ on the nullspace of $\myB$, given by
\begin{equation}\label{def:HeinsNull}
\begin{aligned}
\myHeins^0(\mu) := ~&\{~h \in \myHeins:~\myB(h,\psi) = 0 \quad \forall~\psi \in \myY~\} \\
= ~&\{~(u,y) \in \myU \times \myY:~\mya(y,\psi) = \myb(u,\psi) \quad \forall~ \psi \in \myY ~\} \subset \myHeins.
\end{aligned}
\end{equation}
Note that $\myHeins^0(\mu)$ is a closed subspace of $\myHeins$. For a concise notation, we introduce
\begin{equation}\label{eq:def:Ymu}
\begin{aligned}
\myYmu := ~&\{~y \in \myY:~\exists u \in \myU \text{ s.t. } (u,y) \in \myHeins^0(\mu)~\} \\
= ~&\{~y \in \myY:~\exists u \in \myU \text{ s.t. } \mya(y,\psi) = \myb(u,\psi) \quad \forall~ \psi \in \myY~\}
\end{aligned}
\end{equation}
as the space of all states in $\myHeins^0(\mu)$. Note that $\myYmu \neq \{0\}$ since $\mygammab > 0$ by \eqref{eq:uniformBound:ab}.

With $A((0,\psi),(0,\psi))= 0$ for all $\psi \in \myT ^{\perp}$, $A$ is in general not $\myHeins$-coercive. For the application of Brezzi's Theorem \cite{brezzi}, it is, however, already sufficient for $A$ to be $\myHeins^0(\mu)$-coercive. This property follows from the coercivity of $\mya$ as shown in the following theorem.
\begin{theorem}\label{thm:tCoercivity}
For $\mu \in \mathcal{C}$, let $\mydelta$ be the coercivity constant of $A$ on $\myHeins^0(\mu)$, i.e.
\begin{align} \label{eq:def:delta}
\mydelta ~:=~ \inf _{h \in \myHeins^0(\mu)} \frac{A(h,h)}{~\myNorm[\myHeins]{h}^2}.
\end{align}
Define the ratios
\begin{align}\label{eq:def:eta}
\myetainf := \inf _{(u,y) \in \myHeins^0(\mu)} \frac{\myNormY{y}}{\myNormU{u}} \ge 0
\quad \text{and} \quad
\myetasup := \sup _{(u,y) \in \myHeins^0(\mu)} \frac{\myNormY{y}}{\myNormU{u}} > 0,
\end{align}
and the inf-sup constant
\begin{align} \label{eq:def:kappa}
\mykappa := \inf _{y \in \myYmu}~ \sup _{\tau \in \myT}~ \frac{\mySPY{y}{\tau}}{\myNormY{y}\myNormY{\tau}} = \inf _{y \in \myYmu} \frac{\myNormY{\myProj{\myT}y}}{\myNormY{y}}~\ge~ 0.
\end{align}
Then
\begin{equation}\label{eq:def:deltaLB}
\mydeltaLB(\lambda,\myetainf, \myetasup, \mykappa) :=
\left\{
\begin{array}{ll}
\frac{1 + \lambda \mykappa^2 \myetasup^2 }{1 + \myetasup^2}, & \qquad \text{if } \lambda \mykappa^2 \le 1 \\
\frac{1 + \lambda \mykappa^2 \myetainf^2 }{1 + \myetainf^2}, & \qquad \text{otherwise}
\end{array}
\right.
\end{equation}
is a positive lower bound  to the $\myHeins^0(\mu)$-coercivity constant of $A$, i.e.
\begin{align}\label{eq:bound:deltaLB}
\mydelta
~\ge~ \mydeltaLB(\lambda, \myetainf, \myetasup, \mykappa) ~\ge~ \frac{1}{1+\myetasup^2} ~>~ 0.
\end{align}
\end{theorem}
\begin{proof}
Since $\mya$ is coercive and $\mygammab > 0$ by assumption \eqref{eq:uniformBound:ab}, there exists at least one element $(u,y) \in \myHeins^0(\mu)$ with $y \neq 0$;
hence $\myetasup > 0$.
Let $0 \neq h = (u,y) \in \myHeins^0(\mu)$ be arbitrary.
We note that $y \in \myYmu$.
With the definitions \eqref{eq:def:eta} and \eqref{eq:def:kappa}, we obtain $\myetainf \myNormU{u} \le \myNormY{y} \le \myetasup \myNormU{u}$ and $\myNormY{\myProj{\myT}y} \ge \mykappa \myNormY{y}$ respectively.

If $\lambda \mykappa^2 \le 1$, then
\begin{align*}
A(h,h) &= \myNormU{u}^2 + \lambda \myNormY{\myProj{\myT}y}^2 \\
&\ge \myNormU{u}^2 + \lambda \mykappa^2 \myNormY{y}^2 \\
&\ge (1-x\myetasup^2) ~ \myNormU{u}^2 + (\lambda \mykappa^2 + x)~ \myNormY{y}^2 \\
&\ge \min \{1-x\myetasup^2,~ \lambda \mykappa^2+x\} ~ \myNorm[\myHeins]{h}^2
\end{align*}
for arbitrary $x \in [0,\myetasup^{-2}]$.
The minimum is maximal when both terms equal, which is the case for $x = \frac{1 - \lambda \mykappa^2}{1 + \myetasup^2} < \myetasup^{-2}$.
The minimum then takes the value $\mydeltaLB(\lambda,\myetainf, \myetasup, \mykappa)$.

Now suppose $\lambda \mykappa^2 > 1$.
If $y = 0$, then $\myetainf = 0$ and $A(h,h) =\myNormU{u}^2 = \myNorm[\myHeins]{h}^2$.
For $y \neq 0$, we obtain for any $x \in [0, \lambda \mykappa^2]$:
\begin{align*}
A(h,h) &= \myNormU{u}^2 + \lambda \myNormY{\myProj{\myT}y}^2 \\
&\ge \myNormU{u}^2 + \lambda \mykappa^2 \myNormY{y}^2 \\
&\ge (1+x\myetainf^2) ~ \myNormU{u}^2 + (\lambda \mykappa^2 -x )~ \myNormY{y}^2 \\
&\ge \min \{1+x\myetainf^2,~ \lambda \mykappa^2 - x\} \myNorm[\myHeins]{h}^2.
\end{align*}
With $x = \frac{\lambda \mykappa^2-1}{1 + \myetainf^2}$ we obtain the second part of $\mydeltaLB$.
\qed
\end{proof}

We briefly comment on the meaning of some of the constants. While the inf-sup constant $\mykappa$ relates to the experimental design,  the ratios $\myetainf$ and $\myetasup$ relate to the modelling process. The former depends on the choice of measurements, i.e., the measurement space $\myT$ and reflects how well model modifications can be distinguished based on changes in $\myT$. The latter ratios $\myetainf$ and $\myetasup$ reflect how much change in the state can minimally and maximally be evoked by a model modification in $\myU$, respectively. As already mentioned in the proof of Theorem \ref{thm:tCoercivity}, $\myetasup > 0$ due to $\mygammab > 0$. Note that if $\myetasup = 0$ would hold, $\myU$ would have no effect on the best-knowledge model \eqref{eq:bk}. The ratio $\myetainf$ equals zero if there exists a superfluous search direction $u \in \myU$ whose physical influence $\myb(u,\cdot)$ is not sufficiently captured by $\myY$. To express this correlation, we define, for nontrivial spaces $\mathcal{V} \subset \myU$, $\mathcal{W} \subset \myY$, the inf-sup constant
\begin{align}\label{eq:def:infsup:b}
\beta_b(\mathcal{V},\mathcal{W}) := \inf _{u \in \mathcal{V}} \sup _{w \in \mathcal{W}} \frac{\myb(u,w)}{\myNormU{u} \myNormY{w}} \ge 0.
\end{align}
We can then bound $\myetainf$ by $\frac{\beta_b(\myU,\myYmu)}{\mygammaa} \le \myetainf \le \frac{\beta_b(\myU,\myYmu)}{\myalpha}$. The inf-sup constant $\beta_b(\myU,\myYmu)$ thus provides information on the quality of the model modifications for any parameter $\mu \in \mathcal{C}$ and may prove useful in the design of possible model modifications.

In order to quantify the influence of the noise on the 3D-Var solution, we first present a result concerning the behavior of $\mydelta$.
\blue{
To describe asymptotic behavior in $\lambda$, we make use of the Landau symbol $\Theta$ (see, e.g., \cite{cormen2009introduction}):
A function $f$ lies in the class $\Theta(g)$ if $g$ is an asymptotically tight bound, i.e. there exist $c_1, c_2 > 0$ and a $\lambda_0 \ge 0$ such that $0 \le c_1 g(\lambda) \le f(\lambda) \le c_2 g(\lambda)$ for all $\lambda \ge \lambda_0$.
}

\begin{theorem} \label{thm:equiv}
For $\mu \in \mathcal{C}$, either $\mydelta \le 1$ for all $\lambda > 0$ or $\mydelta \in \Theta(\lambda)$ \blue{with $\lambda_0 > 0$.}
The latter is equivalent to both $\myetainf > 0$ and $\mykappa > 0$ being satisfied.
\end{theorem}
\begin{proof}
By Theorem \ref{thm:tCoercivity} and \eqref{eq:def:gammaAB},
\begin{align} \label{eq:proofThmEquiv:1}
0 < \frac{1}{1+\myetasup^2} \le \mydeltaLB(\lambda,\myetainf, \myetasup, \mykappa) \le \mydelta \le \mygammaA = \max\{1,\lambda\},
\end{align}
where $\mydeltaLB(\lambda,\myetainf, \myetasup, \mykappa) \in \Theta(\lambda)$ if and only if $\myetainf > 0$ and $\mykappa > 0$.
It is hence sufficient to show that $\myalpha \le 1$ for all $\lambda > 0$ when $\myetainf = 0$ or $\mykappa = 0$.
\blue{We note that $\lambda_0 > 0$ stems from the upper bound $\mydelta \le\max\{1,\lambda\}$ in \eqref{eq:proofThmEquiv:1}.}

We suppose first that $\myetainf = 0$. Let $\lambda > 0$ be fixed. With \eqref{eq:def:eta}, there exists a sequence $h_n = (u_n,y_n) \in \myHeins^0(\mu)$, $n \in \mathbb{N}$, such that $\myNormU{u_n} = 1$ and $\myNormY{y_n} \rightarrow 0$ for $n \rightarrow \infty$. Then $\mydelta \le \frac{A(h_n,h_n)}{~\myNorm[\myHeins]{h_n}^2} \le \frac{1 + \lambda \myNormY{y_n}^2}{1 + \myNormY{y_n}^2} \rightarrow 1$ for $n \rightarrow \infty$.

If $\mykappa = 0$, there exists $h_n = (u_n,y_n) \in \myHeins^0(\mu)$, $n \in \mathbb{N}$, with $\myNormY{y_n} = 1$ and, for $n \rightarrow \infty$, $\myNormY{\myProj{\myT}{y_n}} \rightarrow 0$. In particular $\myNormU{u_n}>0$ for all $n \in \mathbb{N}$. Then
\begin{align*}
\mydelta \le
\frac{A(h_n,h_n)}{\myNorm[\myHeins]{h_n}^2}
&= \frac{\myNormU{u_n}^2}{\myNormU{u_n}^2 + \myNormY{y_n}^2} + \lambda \frac{\myNormY{\myProj{\myT}y_n}^2}{\myNormU{u_n}^2 + \myNormY{y_n}^2},
\end{align*}
where the first term is bounded by $(1+\myetainf^2)^{-1} \le 1$ and the second converges to $0$ for $n \rightarrow \infty$ and any fixed $\lambda$.
\qed
\end{proof}

\blue{Theorem \ref{thm:equiv} shows that the lower bound from $\mydeltaLB(\lambda,\myetainf, \myetasup, \mykappa)$ under-predicts the true $\myHeins^0(\mu)$-coercivity constant $\mydelta$ only by a bounded multiplicative factor, i.e. $\mydelta \in \Theta(\mydeltaLB(\lambda,\myetainf, \myetasup, \mykappa))$. We thus gain insight on how the main behaviour of $\mydelta$ is influenced by the quantities $\myetainf$, $\myetasup$, and $\mykappa$. By optimizing these quantities in the modelling process and with the selection of measurements, we have a direct way of influencing $\mydelta$ and thereby the stability of the whole 3D-VAR system \eqref{prob:truth:sp}.
}

We next recall that $\mybeta$ and $\mydelta$ are uniformly bounded away from zero as shown in Theorems~\ref{thm:tInfSup} and \ref{thm:tCoercivity}, respectively. It then follows from Brezzi's Theorem that the 3D-VAR formulation is uniformly well-posed over the parameter domain for all source functions in $\myHeins'$ and $\myHzwei'$~\cite{brezzi,Brezzi_book}.
\blue{For the particular choice in \eqref{eq:defABG}, we get the following stability bounds from Theorem 5.2, p. 38, in \cite{Brezzi_book} after inserting $\myNorm[\myHeins']{F} = \lambda\myNormY{\mydata}$ and $\myNorm[\myHzwei']{\myG} = \myNorm[\myY']{\myfbk}$:}
\begin{theorem}\label{thm:tStability}
Problem \eqref{prob:truth:sp} has a unique solution $(\myheins,\myhzwei) \in \myHeins \times \myHzwei$, and
\begin{subequations}\label{eq:tStability}
\begin{align}
\myNorm[\myHeins]{\myheins} &~~\le~~ \frac{\lambda~\myNormY{\mydata}}{\mydelta} +\frac{\myNorm[\myY']{\myfbk}}{\mybeta}\bigg{(}\frac{\mygammaA}{\mydelta}+1\bigg{)}
\label{eq:tStability:h}\\
\myNorm[\myHzwei]{\myhzwei} &~~\le~~ \frac{\lambda~\myNormY{\mydata}}{\mybeta}\bigg{(}\frac{\mygammaA}{\mydelta}+1\bigg{)}~ +  \frac{\mygammaA~\myNorm[\myY']{\myfbk} }{\mybeta^2} \bigg{(}\frac{\mygammaA}{\mydelta}+1\bigg{)}
\label{eq:tStability:p}
\end{align}
\end{subequations}
\end{theorem}
We note that the effect of experimental noise on the solution is governed by the stability coefficients in front of $\myNormY{\mydata}$. Since $\mygammaA = \max\{1,\lambda\}$, they scale like $\Theta(\lambda / \mydelta)$ in \eqref{eq:tStability:h} and $\Theta(\lambda^2 / \mydelta)$ in \eqref{eq:tStability:p} for $\lambda \ge 1$. Furthermore, we know from Theorem~\ref{thm:equiv} that $\mydelta$ is of order $\Theta(\lambda)$ for $\lambda \rightarrow \infty$ if and only if $\myetainf > 0$ and $\mykappa > 0$. It thus follows that the stability coefficients for the model modification and state variable $\myheins$ are bounded from above independently of $\lambda$ in this case, and that the stability coefficients for the adjoint $\myhzwei$ scales like $\Theta(\lambda)$. If $\myetainf = 0$ or $\mykappa = 0$, the scaling is significantly worse with $\Theta(\lambda)$ and $\Theta(\lambda^2)$, respectively. Provided $\myetainf > 0$, we should thus choose $\myT$ such that the inf-sup constant $\mykappa$ is maximized. We discuss such a stability-based generation of the measurement space in section \ref{sec:greedyOMP}.

\begin{remark}
We briefly comment on the link between the inf-sup constant $\mykappa$, the PBDW formulation, and the 3D-VAR solution.
Suppose $\myYmu$ is closed, as is the case if $\dim \myU < \infty$, and suppose further that $\mykappa > 0$.
Let $\myybk \in \myY$ be the solution of the best-knowledge problem \eqref{eq:bk}.
Then the minimization
\begin{align*}
\min _{y \in \myYmu, d \in \myT} \myNormY{d}^2 \text{ such that } \mySPY{y + d}{\tau} = \mySPY{\mydata - \myybk}{\tau} ~\forall ~\tau \in \myT
\end{align*}
has a unique, $\mu$-dependent solution $(\myyPBDWhat,\myetaPBDW) \in \myYmu \times \myT$.
This problem is the PBDW formulation \cite{PBDW} over the best-knowledge space $\myYmu$, where the basis functions of $\myU$ act as parameters.
For an extensive analysis of state estimation with the PBDW approach, we refer to the literature \cite{PBDW,PBDW2,doi:10.1137/15M1025384,2017arXiv171209594M}.

\blue{
With this notation and the discussion in Remark \ref{rmk:convergence}, the model correction $\myu = \myu(\lambda)$ hence converges for $\lambda \rightarrow \infty$ to the unique element $\myuPBDW$ with minimal norm under all $u \in \myU$ with $(u,\myyPBDWhat) \in \myHeins^0(\mu)$.
}
If the inf-sup condition $\beta_b(\myU,\myYmu) > 0$ holds, \blue{or equivalently} $\myetainf > 0$, \blue{then the limit} $\myuPBDW$ is already uniquely defined by the property $(\myuPBDW,\myyPBDW) \in \myHeins^0(\mu)$.
In either case, we have $\myy = \myy(\lambda) \rightarrow \myyPBDW := \myybk + \myyPBDWhat$ and $\myd(\lambda) = \mydata - \myProj{\myT}\myy(\lambda) \rightarrow \myetaPBDW$ for $\lambda \rightarrow \infty$.

It has been shown in \cite{PBDW,doi:10.1137/15M1025384}, that if $\mydata = \myProj{\myT}\myytrue$, then
\begin{align*}
\myNormY{\myProj{\myYmu}(\myytrue-\myyPBDW)} &\le \frac{1}{\mykappa} \inf_{d \in \myT \cap \myYmu^{\perp}} \inf _{y \in \myybk+\myYmu} \myNormY{\myytrue-(y + d)} \\
\myNormY{\myytrue-(\myyPBDW + \myetaPBDW)} &\le \frac{1}{\mykappa} \inf_{d \in \myT \cap \myYmu^{\perp}} \inf _{y \in\myybk + \myYmu} \myNormY{\myytrue-(y + d)},
\end{align*}
and $\mykappa^{-1}$ is the optimal class performance for the second bound (see \cite{doi:10.1137/15M1025384}).
The approximation of $\myytrue$ through the 3D-VAR solution therefore also profits in the limit $\lambda \rightarrow \infty$ from a preferably large inf-sup constant $\mykappa > 0$.
\end{remark}

\begin{remark}
We note that the PBDW approach requires $\dim \myT \ge \dim \myU$ measurements of the true state for the problem to be well-posed. For the 3D-VAR method, on the other hand, the method remains well-posed for any measurement space --- allthough a ``good'' choice for $\myT$ may increase its stability. Additionally, the regularisation parameter $\lambda$ provides a direct control over the amplification of noise. By changing $\lambda$, the 3D-VAR method can be adjusted for different noise levels. A comparatively small $\lambda$-value may result in better state estimate than a large $\lambda$-value given noisy data.
\end{remark}

\subsection{Numerical Solution}
For the numerical solution of the 3D-VAR saddle-point system \eqref{prob:truth:sp}, the spaces $\myU$ and $\myY$ are replaced with finite-dimensional approximation spaces, e.g. conforming finite element spaces. Let $(\phi_m)_{m=1}^{\myM} \subset \myU$ and $(\psi_n)_{n=1}^{\myN} \subset \myY$ be bases of $\myU$ and $\myY$, respectively. For $\myT$ we consider the basis $(\tau_l)_{l=1}^L$ which can be represented in terms of the basis of $\myY$.

The 3D-VAR solution $(\myheins,\myhzwei) = (\myu,\myy,\myw) \in \myU \times \myY \times \myY$ is then uniquely defined by the basis coefficients of the model modification, state, and adjoint arguments represented by $\myu = \sum _{m = 1} ^{\myM} \myvecu_m \phi_m$, $\myy = \sum _{n = 1} ^{\myN} \myvecy_n \psi_n$, and $\myw = \sum _{n = 1} ^{\myN} \myvecw_n \psi_n$, respectively.
The coefficient vector $(\myvecu,\myvecy,\myvecw) \in \mathbb{R}^{\myM + 2\myN}$ is the unique solution of the linear $(\myM+2\myN)\times(\myM+2\myN)$ saddle-point system
\begin{equation} \label{eq:numImp:t}
\begin{aligned}
\left(
\begin{array}{ccc}
\textbf{U} & \textbf{0} & ~\textbf{B}^T(\mu)~ \\
\textbf{0} & \lambda\textbf{P} & \textbf{A}^T(\mu) \\
~\textbf{B}(\mu)~ & ~\textbf{A}(\mu)~ & \textbf{0}
\end{array}
\right)
\left(
\begin{array}{c}
\myvecu \\
\myvecy \\
\myvecw
\end{array}
\right)
=
\left(
\begin{array}{c}
\textbf{0}\\
\lambda \myvecdata \\
\myvecfbk(\mu)
\end{array}
\right).
\end{aligned}
\end{equation}
Here, $\mathbf{U} \in \mathbb{R}^{\myM \times \myM}$ is the mass matrix on $\myU$, $\mathbf{P} \in \mathbb{R}^{\myN \times \myN}$ is the matrix representation of the state projection $\myProj{\myT}$, $\mathbf{A}(\mu) \in \mathbb{R}^{\myN \times \myN}$ is the stiffness matrix, and $\mathbf{B}(\mu) \in \mathbb{R}^{\myN \times \myM}$ is the model modification matrix.
The entries are given by $\mathbf{U}_{i,j} = \mySPU{\phi_j}{\phi_i}$, $\mathbf{P}_{i,j} = \mySPY{\myProj{\myT} \psi_j}{\myProj{\myT} \psi_i}$, $\mathbf{A}_{i,j}(\mu) = \mya(\psi_j,\psi_i)$ and $\mathbf{B}_{i,j}(\mu) = -\myb(\phi_j,\psi_i)$.
On the right side of the equation, we have the best-knowledge source term vector $\myvecfbk(\mu) \in \mathbb{R}^{\myN}$ given by $\mathbf{f}_{{\rm{bk}},i}(\mu) = \myfbk(\psi_i)$, and the data vector $\myvecdata \in \mathbb{R}^{\myN}$ with $\mathbf{y} _{{\rm{d}},i} = \mySPY{\mydata}{\psi_i}$.
The latter can be obtained from the measurement data $\myvecmeasured$:
For $\mathbf{T} \in \mathbb{R}^{L \times L}$, $\mathbf{T}_{i,j} := g_i(\tau_j)$ and $\mathbf{S}\in \mathbb{R}^{\myN \times L}$, $\mathbf{S}_{i,j}:= \mySPY{\tau_j}{\psi_i}$, we have $\myvecdata = (\mathbf{S}\mathbf{T}^{-1}) \myvecmeasured$.

We note that the affine nature \eqref{eq:affine} of the bilinear and linear forms $\mya$, $\myb$ and $\myfbk$ transfers to the corresponding matrices $\mathbf{A}(\mu)$, $\mathbf{B}(\mu)$ and $\myvecfbk(\mu)$. Hence, for different parameters they can be assembled from precomputed, $\mu$-independent matrices, e.g. the stiffness matrix $\textbf{A}(\mu)$ has the form $\textbf{A}(\mu) = \sum _{\vartheta = 1} ^{Q_a} \theta_a^{\vartheta}(\mu) \textbf{A}^{\vartheta} \in \mathbb{R}^{\myN \times \myN} \text{ with } (\textbf{A}^{\vartheta})_{i,j} = a^{\vartheta}(\Psi_j,\Psi_i) \text{ for } \vartheta = 1,...,Q_a$.


\section{RB Approximation}
\label{sec:RB}
The stability analysis in section \ref{sec:tStability}, in particular Theorem \ref{thm:tStability}, indicates, that the 3D-VAR problem \eqref{prob:truth} is well-posed for every parameter $\mu \in \mathcal{C}$ with uniformly bounded stability constants. Yet, the implementation of the 3D-VAR method involves solving a large linear system of equations  for each parameter $\mu$, where the $\mu$-dependent parts of the block matrix need to be assembled anew whenever $\mu$ changes. For real-time and many-query applications over the parameter domain, the computational cost hence becomes large.

We address these problems by introducing an RB scheme: We first derive an RB formulation of the 3D-VAR method and employ the extensive stability analysis of section \ref{sec:tStability} to the RB spaces to emphasize the connection to the measurement space. We comment briefly on the numerical implementation of the RB method, before we provide an error analysis between the solutions of the original (truth) 3D-VAR problem \eqref{prob:truth} and its RB approximation. We derive computationally efficient a posteriori error bounds for the error in the model modification, state estimate, adjoint solution and the misfit between observations in the measurement space.

\blue{We briefly comment on the notation. We require several reduced basis quantities in this section which are direct analogues of their truth counterparts, e.g., of the inf-sup constant $\beta_{\cal T}(\mu)$. In order to avoid various repetitions, quantities with an indexed $()_{\cal{R}}$ are defined as before with the FE spaces replaced by the corresponding RB spaces.}

%
\subsection{RB Problem Statement}
\label{sec:rbStability}

The main contributions to the computational cost of the 3D-VAR method \eqref{prob:truth} come from the evaluation of the PDE over the state space $\myY$ in the direction of all possible model modifications $\myU$.
If the amount of model modifications is restricted to the most essential directions, and the solution of the PDE is approximated efficiently with a reduced order model, then the 3D-VAR saddle-point system reduces in size and may be solved faster. To this end, we assume to be given closed, non-trivial subspaces $\myURB \subset \myU$ and $\myYRB \subset \myY$, which we refer to as RB spaces. In practice, $\myURB$ and $\myYRB$ are typically low-dimensional.
In section \ref{sec:spaces}, we discuss how they can be chosen based on the intrinsic structure of the PDE. Note that we do not consider a reduction of the measurement space $\myT$, i.e. the RB problem is based on the same data and no measurements are ignored.

Our reduced-basis 3D-VAR problem has then the form
\begin{equation}\label{prob:RB}
\begin{aligned}
&\min _{(u_{\rm{R}},y_{\rm{R}},d) \in \myURB \times \myYRB \times \myT} \frac{1}{2} \myNormU{u_{\rm{R}}}^2 + \frac{\lambda}{2} \myNormY{d}^2 \text{ s.t. } \\
&\left.
\begin{array}{rll}
\mya(y_{\rm{R}},\psi_{\rm{R}}) &= \myfbk (\psi_{\rm{R}}) + \myb(u_{\rm{R}},\psi_{\rm{R}}) & \quad \forall~\psi_{\rm{R}} \in \myYRB \\
\mySPY{y_{\rm{R}}+d}{\tau} &= \mySPY{\mydata}{\tau} & \quad \forall ~\tau \in \myT.
\end{array}
\right.
\end{aligned}
\end{equation}
Similar to before, the reduced basis minimization \eqref{prob:RB} can be written as an equivalent saddle-point problem over the spaces $\myHeinsRB := \myURB \times \myYRB$ and $\myYRB$:  Given $\mu \in \mathcal{C}$, find $(\myheinsRB,\myhzweiRB) = ((\myuRB,\myyRB),\mywRB) \in \myHeinsRB \times \myYRB$ such that
\begin{equation}\label{prob:RB:sp}
\begin{aligned}
A(\myheinsRB,k_{\rm{R}}) + \myB(k_{\rm{R}},\myhzweiRB) &= F(k_{\rm{R}}) && \forall~k_{\rm{R}} \in \myHeinsRB  \\
\myB(\myheinsRB,\psi_{\rm{R}}) &= \myG(\psi_{\rm{R}}) && \forall~\psi_{\rm{R}} \in \myYRB.
\end{aligned}
\end{equation}
The forms $A$, $\myB$, $F$ and $\myG$ are the ones defined in \eqref{eq:defABG} restricted to the respective spaces.

Following standard nomenclature in the RB literature, we distinguish between the original 3D-VAR formulation \eqref{prob:truth} and its RB approximation \eqref{prob:RB} by referring to the former as the truth problem and the latter as the RB problem. On a structural level, the only formal difference between the truth and the RB problem is that $\myT$ is a closed subspace of the state space $\myY$, while in general $\myT \not \subset \myYRB$.
However, this property remained unused within the stability analysis in section \ref{sec:tStability}, which therefore directly applies to the RB case when $\myU$ and $\myY$ are formally replaced with $\myURB$ and $\myYRB$. We therefore omit a detailed stability analysis as well as the {\it a priori} error analysis. The latter directly follows from standard a-priori theory for Galerkin projections of saddle-point problems \cite{brezzi,Brezzi_book}. More specifically, in the present case we can even apply the results for symmetric problems with a coercive bilinear form $A$ from~\cite{Gerner_2012,Gerner_phd}.

\subsection{Computational Procedure for the RB System}

In the finite-dimensional setting and after a preparatory offline phase, the RB problem \eqref{prob:RB:sp} can be solved independently of the dimensions $\myN = \dim \myY$ and $\myM = \dim \myU$ of the high-fidelity spaces.
For given basis functions $(\phi_{\rm{R,} m})_{m=1}^M$ of $\myURB$ and $(\psi_{\rm{R,} n})_{n=1}^N$ of $\myYRB$, the basis coefficients $(\myvecuRB,\myvecyRB,\myvecwRB) \in \mathbb{R}^{M+2N}$ of $(\myheinsRB,\myhzweiRB) = (\myuRB,\myyRB,\mywRB) \in \myU \times \myY \times \myY$ can be computed by solving the linear system
\begin{equation} \label{eq:RB:numImp}
\begin{aligned}
\left(
\begin{array}{ccc}
\textbf{U}_{\rm{R}} & \textbf{0} & ~\textbf{B}_{\rm{R}}(\mu)~ \\
\textbf{0} & \lambda\textbf{P}_{\rm{R}} & \textbf{A}_{\rm{R}}^T(\mu) \\
~\textbf{B}_{\rm{R}}^T(\mu)~ & ~\textbf{A}_{\rm{R}}(\mu)~ & \textbf{0}
\end{array}
\right)
\left(
\begin{array}{c}
\myvecuRB \\
\myvecyRB \\
\myvecwRB
\end{array}
\right)
=
\left(
\begin{array}{c}
\textbf{0}\\
\lambda \myvecdataRB \\
\myvecfbkRB(\mu)
\end{array}
\right).
\end{aligned}
\end{equation}
Once $\myURB$ and $\myYRB$ have been chosen, e.g. with the approaches presented in section \ref{sec:spaces}, the matrices and vectors can be obtained from the previous ones in \eqref{eq:numImp:t} by multiplication with the basis representations of $(\phi_{\rm{R,}m})_{m=1}^{\myM}$ in $(\phi_m)_{m=1}^M$, and $(\psi_{\rm{R,}n})_{n=1}^N$ in $(\psi_n)_{n=1}^{\myN}$.
This computation needs only to be performed once, since the $\mu$-dependent forms $\textbf{A}_{\rm{R}}(\mu)$, $\textbf{B}_{\rm{R}}(\mu)$ and $\myvecfbkRB(\mu)$ can be assembled from small, stored, $\mu$-independent matrices in $\mathcal{O}(Q_aN^2+Q_bMN+Q_fN)$ within the online phase thanks to the affine decomposition~\eqref{eq:affine}.

For any measurement data $\myvecmeasured \in \mathbb{R}^L$, the vector $\myvecdataRB$ can be obtained via $\myvecdataRB = (\mathbf{S}_{\rm{R}}\mathbf{T}^{-1}) \mathbf{m}$, where $\mathbf{T} \in \mathbb{R}^{L \times L}$, $\mathbf{T}_{i,j} := g_i(\tau_j)$ and $\mathbf{S}_{\rm{R}}\in \mathbb{R}^{N \times L}$, $(\mathbf{S}_{\rm{R}})_{i,j}:= \mySPY{\tau_j}{\psi_{\rm{R},i}}$ for the basis $(\tau_l)_{l=1}^L$ of $\myT$.
Assuming that $\mathbf{S}_{\rm{R}}\mathbf{T}^{-1}$ is precomputed, each new measurement set $\myvecmeasured$ necessitates computational cost of order $\mathcal{O}(NL)$.
Once assembled, the saddle-point system \eqref{eq:RB:numImp} can be solved in $\mathcal{O}((M+2N)^3)$.
Altogether, the online solves are completely independent of the dimensionality of the truth spaces.


\subsection{{\it A Posteriori} Error Estimation}
\label{sec:aposteriori}

Given the truth and RB saddle-point problems \eqref{prob:truth:sp} and \eqref{prob:RB:sp}, we may follow the approach in \cite{Gerner_2012,gerner_siam_sp} to derive {\it a posteriori} error bounds for $\myheins-\myheinsRB$ and $\myhzwei-\myhzweiRB$.
However, these bounds provide no individual information on the error between the model corrections $\myu$ and $\myuRB$, or between the state estimates $\myy$ and $\myyRB$.
We therefore pursue the approach from \cite{Kaercher_optCon} to develop separate error bounds for the errors $\myeu := \myu - \myuRB \in \myU$ in the model correction, $\myey := \myy - \myyRB \in \myY$ in the state estimate, and $\myew := \myw - \mywRB \in \myY$ in the adjoint solution.
In addition, we provide an error bound for the difference $\myed := \myd - \mydRB \in \myT$ between the misfits $\myd := \mydata - \myProj{\myT}\myy \in \myT$ and $\mydRB := \mydata - \myProj{\myT}\myyRB \in \myT$.

We first define the residual functions
\begin{equation} \label{eq:def:r}
\begin{aligned}
&\myru : \myU \rightarrow \mathbb{R}: &&r_u(\phi) := \myb(\phi,\mywRB) - \mySPU{\myuRB}{\phi}, \\
&\myrp : \myY \rightarrow \mathbb{R}: &&r_p(\psi) := \lambda \mySPY{\psi}{\mydRB} - \mya(\psi,\mywRB), \\
&\myry : \myY \rightarrow \mathbb{R}: &&r_y(\psi) := \myfbk(\psi) + \myb(\myuRB,\psi) - \mya(\myyRB,\psi),
\end{aligned}
\end{equation}
all three of which are linear, continuous and $\mu$-dependent.
By subtracting \eqref{prob:RB:sp} from \eqref{prob:truth:sp} and varying over the whole spaces $\myU$, $\myY$, $\myT$ individually, the error tuple $(\myeu,\myey,\myed,\myew) \in \myU \times \myY \times \myY \times \myT$ solves the variational system
\begin{subequations}\label{eq:errPDE}
\begin{align}
\mySPU{\myeu}{\phi} - \myb(\phi,\myew) &= \myru(\phi) && \forall ~ \phi \in \myU, \label{eq:errPDE:1}\\
\mya(\psi,\myew) - \lambda \mySPY{\myed}{\psi} &= \myrp(\psi) && \forall ~\psi \in \myY, \label{eq:errPDE:2}\\
\mya(\myey,\psi) - \myb(\myeu,\psi) &= \myry(\psi) && \forall ~ \psi \in  \myY, \label{eq:errPDE:3}\\
\mySPY{\myey}{\tau} + \mySPY{\myed}{\tau} &= 0 && \forall ~ \tau \in \myT. \label{eq:errPDE:4}
\end{align}
\end{subequations}
The first three equations provide alternative representations of $\myru$, $\myrp$ and $\myry$, which can be used to bound their norms in regard to the approximation errors by
\begin{equation}\label{eq:normbound:r}
\begin{aligned}
\myNreins &~~\le~~ \myNormU{\myeu} + \mygammab ~ \myNormY{\myew} \\
\myNrzwei &~~\le~~ \mygammaa~\myNormY{\myew} + \lambda \myNormY{\myed} \\
\myNrdrei &~~\le~~ \mygammaa~\myNormY{\myey} + \mygammab~\myNormU{\myeu}.
\end{aligned}
\end{equation}
Hence, small errors are reflected by the residual norms.
In case the truth and the RB solution coincide, $\myru$, $\myrp$ and $\myry$ equal zero. Following the alternative approach in \cite{Kaercher_optCon}, we obtain the following:

\begin{theorem} \label{thm:aPos:ErrB}
Let $\mu \in \mathcal{C}$ be given and define
%
%
%
%
\begin{equation*}
\begin{aligned}
p_u &~:=~ \myNreins + \frac{\mygammab}{\myalpha} ~\myNrzwei &&
q_u ~:=~ \frac{2}{\myalpha} \myNrzwei\myNrdrei +  \frac{\lambda}{4 \myalpha^2} \myNrdrei^2 \\
p_d &~:=~ \frac{1}{\myalpha}\myNrdrei &&
q_d ~:=~ \frac{2}{\lambda\myalpha} \myNrzwei\myNrdrei + \frac{1}{4\lambda}p_u^2.
\end{aligned}
\end{equation*}
Then for the unique solution $(\myeu,\myey,\myed,\myew) \in \myU \times \myY \times \myY \times \myT$ of \eqref{eq:errPDE}, we have
\begin{equation} \label{eq:apos:indiv}
\begin{aligned}
\myNormU{\myeu} &~\le~ \frac{1}{2} p_u + \sqrt{\frac{1}{4}p_u^2 + q_u} &&
\quad \myNormY{\myey} ~\le~ \frac{1}{\myalpha}\myNrdrei  + \frac{\mygammab}{\myalpha}  \myNormU{\myeu} \\
\myNormY{\myed} &~\le~  \frac{1}{2} p_d + \sqrt{\frac{1}{4}p_d^2 + q_d} &&
\quad \myNormY{\myew} ~\le~ \frac{1}{\myalpha} \myNrzwei + \frac{\lambda}{\myalpha}  \myNormY{\myed}.
\end{aligned}
\end{equation}
\end{theorem}

\begin{proof}
To bound $\myNormY{\myew}$, we use the coercivity \eqref{eq:def:alpha} of $\mya$, and set $\psi := \myew$ in \eqref{eq:errPDE:2} to obtain
\begin{equation} \label{eq:apos:proof:w}
\begin{aligned}
\myalpha \,  \myNormY{\myew} ~\le~ \frac{\mya (\myew,\myew)}{\myNormY{\myew}} ~=~ \frac{\myrp(\myew) + \lambda \mySPY{\myed}{\myew}}{\myNormY{\myew}} ~\le~ \myNrzwei  + \lambda \, \myNormY{\myed}.
\end{aligned}
\end{equation}
The bound for $\myNormY{e_p}$ is derived similarly with $\psi := \myey \in \myY$ in equation \eqref{eq:errPDE:3}; we thus get
\begin{equation} \label{eq:apos:proof:y}
\begin{aligned}
\myalpha \,  \myNormY{\myey} \le \frac{\mya (\myey,\myey)}{\myNormY{\myey}}
= \frac{\myry(\myey) + \myb(\myeu,\myey)}{\myNormY{\myey}}
\le \myNrdrei + \mygammab \, \myNormU{\myeu}.
\end{aligned}
\end{equation}
For $\myNormU{\myeu}$ and $\myNormY{\myed}$, we start with the choice $\phi = \myeu \in \myU$ in \eqref{eq:errPDE:1},
\begin{equation} \label{eq:apos:proof:u}
\begin{aligned}
\myNormU{\myeu}^2 &\stackrel{\eqref{eq:errPDE:1}}{~~=~~} \myru(\myeu) + \myb(\myeu,\myew) \\
&\stackrel{\eqref{eq:errPDE:3}}{~~=~~} \myru(\myeu) - \myry(\myew) + \mya(\myey,\myew) \\
&\stackrel{\eqref{eq:errPDE:2}}{~~=~~} \myru(\myeu) - \myry(\myew) + \myrp(\myey) + \lambda \mySPY{\myed}{\myey} \\
&\stackrel{\eqref{eq:errPDE:4}}{~~=~~}  \myru(\myeu) - \myry(\myew) + \myrp(\myey) - \lambda \myNormY{\myed}^2 \\
&\stackrel{~}{~~\le ~~} \myNreins \myNormU{\myeu} + \myNrdrei \myNormY{\myew} + \myNrzwei \myNormY{\myey} - \lambda \myNormY{\myed}^2.
\end{aligned}
\end{equation}
The use of the previous bounds \eqref{eq:apos:proof:w} and \eqref{eq:apos:proof:y} in the second and third term yields
\begin{equation} \label{eq:apos:proof:res}
\begin{aligned}
&\myNormU{\myeu}^2 + \lambda \myNormY{\myed}^2 \\
& ~\le~ (\myNreins +  \frac{ \mygammab}{\myalpha}~\myNrzwei )\myNormU{\myeu}+\frac{ 2 \myNrzwei\myNrdrei}{\myalpha} +  \frac{\lambda\myNrdrei }{\myalpha}\myNormY{\myed}.
\end{aligned}
\end{equation}
For the last term, we have $\frac{\lambda}{\myalpha}\myNrdrei \myNormY{\myed} \le \frac{\lambda}{4 \myalpha^2} \myNrdrei^2 + \lambda \myNormY{\myed}^2$ by Young's Inequality.
After subtracting $\lambda \myNormY{\myed}^2$ from either side of \eqref{eq:apos:proof:res}, we have
\begin{align*}
\myNormU{\myeu}^2 ~~\le~~ (\myNreins + \frac{\mygammab}{\myalpha}~\myNrzwei) \myNormU{\myeu} +  \frac{2}{\myalpha} \myNrzwei\myNrdrei +  \frac{\lambda \myNrdrei^2}{4 \myalpha^2},
\end{align*}
which yields the bound for $\myNormU{e_u}$ with the quadratic formula.
For the bound of $\myNormY{\myed}$ we use Young's Inequality on the first term in the inequality \eqref{eq:apos:proof:res} to get
\begin{align*}
(\myNreins + \frac{\mygammab}{\myalpha} ~\myNrzwei) \myNormU{\myeu}
~~\le ~~
\frac{1}{4} (\myNreins + \frac{\mygammab}{\myalpha}~\myNrzwei)^2 + \myNormU{\myeu}^2.
\end{align*}
Then, $\myNormU{\myeu}^2$ can be subtracted from either side of \eqref{eq:apos:proof:res}, which leaves
\begin{align*}
\lambda \myNormY{\myed}^2 \le  \frac{\lambda \myNrdrei}{\myalpha} \myNormY{\myed} +  \frac{2}{\myalpha} \myNrzwei\myNrdrei + \frac{1}{4} (\myNreins + \frac{\mygammab}{\myalpha} ~\myNrzwei)^2.
\end{align*}
The bound for $\myNormY{e_d}$ can now also be obtained with the quadratic formula.
\qed
\end{proof}
The error bounds in Theorem \ref{thm:aPos:ErrB} can be used in the finite-dimensional setting to bound the approximation error without computing the truth solution.
First note that the bounds decrease with the coercivity constant $\myalpha$, which can therefore be substituted with its efficiently computable lower bound $\myalphaLB \le \myalpha$.
Analogously, $\mygammab$ can be replaced with its upper bound $\mybUB$.
Due to the affine parameter dependence assumed in \eqref{eq:affine}, the residual norms can be computed from \eqref{eq:def:r} in an offline-online-procedure independent of the space dimensions $\myN = \dim \myY$ and $\myM = \dim \myU$.
Their computation is a standard procedure in the RB literature and therefore omitted, see e.g. \cite{RB_Patera_2007}.
To summarize, we then obtain {\it a posteriori} bounds $\myDu$, $\myDy$, $\myDd$ and $\myDw$ for $\myNormU{\myeu}$, $\myNormY{\myey}$, $\myNormY{\myed}$, and $\myNormY{\myew}$ that may be computed fast and efficiently. We note that the boundedness \eqref{eq:normbound:r} of the residual norms with respect to the errors ensures a correspondence between the {\it a posteriori} error bounds and the behavior of the true error.


\section{Space Construction}
\label{sec:spaces}

In the previous sections, we have assumed the measurement space $\myT$ and the RB spaces $\myURB$ and $\myYRB$ to be given. In this section, we present ideas on how the previously discussed properties concerning stability and approximation quality may be used for constructing the spaces. More specifically, we first introduce a greedy-OMP algorithm for the stability-based generation of $\myT$, and subsequently present an approach for the construction of the reduced basis spaces.
\subsection{Greedy-OMP Algorithm} \label{sec:greedyOMP}
In some applications, prior information can allow for a low-dimensional model correction space $\myU$.
One reason would be that, due to the known presence of noise, the focus of interest lies in the large-scale behaviour of the model correction rather than noise-sensitive details and oscillations.
Such a low-dimensional choice can be of great benefit, since it can counteract noise amplification already for low-dimensional measurement spaces, as indicated by our stability analysis in section \ref{sec:tStability}.
In the following, we propose a greedy-OMP algorithm that chooses a measurement space $\myT$ such that 1) the influence of noise upon the 3D-VAR model correction and state estimates is bounded independently of $\lambda$ and $\mu$, and 2)
solutions of the modified model \eqref{eq:forward} are distinguishable for different parameters.

We presuppose that $M := \dim \myU$ is small and that $\myY$ can sufficiently capture the modifications induced by $\myU$ so that $\myetainf > 0$ holds uniformly on $\mathcal{C}$.
Then, by the stability analysis in section \ref{sec:tStability}, the amplification of noise is bounded for the 3D-VAR model correction and state solution $\myheins$ $\lambda$-independently if $\mykappa > 0$.
For any fixed parameter $\mu$, we can compute the space $\myYmu$ of state modifications
by considering the forward problem
\begin{align}\label{eq:greedyOMP:1}
\text{For }u_0 \in \myU \text{ find } y_{\mu}(u_0)\in \myY \text{ s.t. }\mya(y_{\mu}(u_0),\psi) = \myb(u_0,\psi) \quad \forall ~\psi \in \myY.
\end{align}
Then $\myYmu = \text{span}(y_{\mu}(\phi_m))_{m=1}^M$ for any basis $(\phi_m)_{m=1}^M$ of $\myU$.

We assume to be given a library $\mathcal{L} \subset \myY'$, from which we may select different measurement functionals $g_l \in \mathcal{L}$ whose Riesz representation then spans $\myT$, see \eqref{eq:def:T}.
For any fixed parameter $\mu$, a generalized OMP algorithm can be applied to the space $\myYmu$ to iteratively expand $\myT$ and --- under some restrictions upon the library --- enforce $\mykappa > 0$  for $L = \dim \myT \ge M = \dim \myYmu$ (see \cite{Olga}).
Theoretically, $\myT$ can be expanded by repeating this procedure for sufficiently many parameters.
However, the necessity to solve the forward problem \eqref{eq:greedyOMP:1} for $u_0 = \phi_m$, $m=1,...,M$ and each parameter accumulates a high computational cost.

We can improve this procedure with our RB approximation:
We set $\myURB :=\myU$ and assume to have an RB space $\myYRB \subset \myY$ with $\myetainfRB > 0$ that sufficiently approximates $\myYmu$ in the following sense:
For any $\mu \in \mathcal{C}$ there exists an $\varepsilon_{\mu}$, $0 \le \varepsilon_{\mu} \ll 1$, with:
\begin{equation} \label{eq:YmuApproxProperty}
\text{If } (u,y) \in \myHeins^0(\mu) \text{ and } (u,y_{\rm{R}}) \in \myHeinsRB^0(\mu) \text{ then } \myNormY{y - y_{\rm{R}}} \le \varepsilon_{\mu} \myNormY{y}.
\end{equation}
We emphasize that the RB spaces here may be considered as completely separate from the RB 3D-VAR method. We use the same notation as in section \ref{sec:RB} only for simplicity, as it avoids the needless repetition of definitions and results.

For ${u_0} \in \blue{\myU = \myURB}$, let $y({u_0}) \in \myYmu$ and $y_{\rm{R}}({u_0})\in \myYmuRB$ denote the unique elements with $({u_0},y({u_0})) \in \myHeins^0(\mu)$ and $({u_0},y_{\rm{R}}({u_0})) \in \myHeinsRB^0(\mu)$ respectively.
Then
\begin{align}\label{eq:greedyOMP:2}
\frac{\myNormY{y_{\rm{R}}({u_0})}}{\myNormY{y({u_0})}} \ge
\frac{\myNormY{y({u_0})} - \myNormY{y({u_0})-y_{\rm{R}}({u_0})}}{\myNormY{y({u_0})}} \ge
1-\varepsilon_{\mu}.
\end{align}
The inf-sup constant $\mykappa$ on the high fidelity space $\myY$ can then be bounded with respect to $\mykappaRBmu$ on the RB space $\myYRB$ via
\begin{align*}
\mykappa =& \inf _{{u_0} \in \blue{\myU}} \sup _{\tau \in \myT}\frac{\mySPY{y({u_0})}{\tau}}{\myNormY{y({u_0})}\myNormY{\tau}} \\
=& \inf _{{u_0} \in \blue{\myU}} \sup _{\tau \in \myT} \frac{\mySPY{y_{\rm{R}}({u_0})}{\tau}}{\myNormY{y({u_0})} \myNormY{\tau}}  + \frac{\mySPY{y({u_0})-y_{\rm{R}}({u_0})}{\tau}}{\myNormY{y({u_0})} \myNormY{\tau}}\\
\ge & \inf _{{u_0} \in \blue{\myU}}  \frac{\myNormY{y_{\rm{R}}({u_0})}}{\myNormY{y({u_0})}}\sup _{\tau \in \myT} \frac{\mySPY{y_{\rm{R}}({u_0})}{\tau}}{\myNormY{y_{\rm{R}}({u_0})} \myNormY{\tau}} - \varepsilon_{\mu} \\
\ge & ~(1-\varepsilon_{\mu}) \inf _{{u_0} \in \blue{\myU = \myURB}} \sup _{\tau \in \myT} \frac{\mySPY{y_{\rm{R}}({u_0})}{\tau}}{\myNormY{y_{\rm{R}}({u_0})} \myNormY{\tau}} - \varepsilon_{\mu}\\
= & ~(1-\varepsilon_{\mu}) \mykappaRBmu - \varepsilon_{\mu}.
\end{align*}
Hence, $\myT$ is a stabilizing choice for the truth 3D-VAR method, if the inf-sup condition $\mykappaRBmu > 0$ holds and $\varepsilon_{\mu}$ is sufficiently small.

The greedy-OMP algorithm \ref{alg:greedyOMP} generates a measurement space $\myT$ from the library $\mathcal{L} \subset \myY'$ to increase $\mykappaRBmu$ over a training set $\Xi _{\rm{train}} \subset \mathcal{C}$.
The initial computation of $\myYmuRB$ for $\mu \in \Xi _{\rm{train}}$ \blue{(line \ref{line:compYmuRB}), and the searches for $y_L$ (line \ref{line:OMP1}) and $\mu_{L+1}$ (line \ref{line:findNext})} in each iteration requires only online operations, i.e., independent of the FE dimensions.
The selection of the measurement functional \blue{$g_L$ (line \ref{line:OMP2})} depends on \blue{the FE-dimension} $\myN = \dim \myY$ and the size of the library, \blue{and the expansion of $\myT$ (line \ref{line:expansion}) needs $\myN$-dependent computations for the Riesz representation and orthonormalization.
However, both (lines \ref{line:OMP2} and \ref{line:expansion})} need to be performed only once per iteration.
In the algorithm, \blue{the measurement functional $g_L$ is chosen in lines \ref{line:OMP1} and \ref{line:OMP2} so that the expansion of $\myT$ in line \ref{line:expansion} increases $\mykappaRBmu[\mu_{L}]$.
This is an application to the space $\myYmuRB[\mu_L]$ of the worst-case OMP algorithm in \cite{Olga}, which is in turn a greedy generalization of classical OMP algorithms.}
Other selection strategies for increasing the inf-sup constant with regard to measurement functionals \blue{exist and can be substituted in lines \ref{line:OMP1}-\ref{line:OMP2}; we refer to \cite{Olga} and \cite{PBDW} for further options and both analytical and numerical comparisons.}

\begin{algorithm}
\caption{Greedy Orthogonal Matching Pursuit}\label{alg:greedyOMP}
\begin{algorithmic}[1]
\Require $\Xi _{\rm{train}} \subset \mathcal{C}$ training set, $\mathcal{L} \subset \myY'$ library, $\beta_0> 0$ target value, $\mu_1 \in \Xi _{\rm{train}}$, $L_{\rm{max}}$
\State	Compute $\myYmuRB$ for each $\mu \in \Xi _{\rm{train}}$ \label{line:compYmuRB}
\State 	$\myT \gets \{0\}$, $L \gets 0$, $\beta \gets 0$
\While	{$\beta \le \beta_0$ \textbf{ and } $L \le L_{\rm{max}}$}
\State	$L \gets L+1$
\State   \blue{$y_L \gets \text{arg} \max \{\myNormY{y - \myProj{\myT}y}:~ y \in \myYmuRB[\mu_L], \myNormY{y} = 1 \}$} \label{line:OMP1}
\State   \blue{$g_L \gets \text{arg} \max _{g \in \mathcal{L}} |g(y_L - \myProj{\myT}y_L)| / \myNorm[\myY']{g}$} \label{line:OMP2}
\State   Expand $\myT$ with the Riesz representation of $g_L$ \label{line:expansion}
\State  	Find $\mu_{L+1} \in \text{arg} \min _{\mu \in \Xi _{\rm{train}}} \mykappaRBmu$ \label{line:findNext}
\State   $\beta \gets \mykappaRBmu[\mu_{L+1}]$
\EndWhile
\end{algorithmic}
\end{algorithm}

If the 3D-VAR method is to be employed for parameter estimation, then $\myT$ should be able to distinguish between solutions of the modified model \eqref{eq:forward} for different parameters.
To this end, we define, for $\mu,\nu \in \mathcal{C}$,
\begin{align*}
\beta_{\myT}(\mu,\nu) := \inf _{y \in \myYmu[(\mu,\nu)]} \sup _{\tau \in \myT} \frac{\mySPY{y}{\tau}}{\myNormY{y}\myNormY{\tau}}
\text{~ with ~}
\myYmu[(\mu,\nu)] := \text{span} \{~\myybk[\mu],\myybk[\nu],\myYmu[\mu],\myYmu[\nu]~\}.
\end{align*}
Here, $\myybk[\mu]$ and $\myybk[\nu]$ are the best-knowledge solutions of \eqref{eq:bk} for the respective parameters.
If $y_{\mu} = y_{\mu}(u_1)$ and $y_{\nu} = y_{\nu}(u_2)$ are solutions of the modified model \eqref{eq:forward} for parameters $\mu,\nu \in \mathcal{C}$ and modifications $u_1,u_2 \in \myU$, then $y_{\mu},y_{\nu} \in \myYmu[(\mu,\nu)]$.
The experimentally-observable difference $\myNormY{\myProj{\myT}(y_{\mu}-y_{\nu})} \ge \beta_{\myT}(\mu,\nu)\myNormY{y_{\mu}-y_{\nu}}$ is hence bounded from below relative to their actual distance.
For parameter estimation, $\myT$ should thus be chosen such that $\beta_{\myT}(\mu,\nu) > 0$ on $\mathcal{C} \times \mathcal{C}$.
We can generate $\myT$ similarly to before, with a slight modification of the greedy-OMP algorithm \ref{alg:greedyOMP}:
In addition to $\mykappaRBmu$, we consider
\begin{align}\label{eq:def:kappaRB:mu:nu}
\beta_{\myT\rm{,R}}(\mu,\nu) := \inf _{y \in \myYmuRB[(\mu,\nu)]} \sup _{\tau \in \myT} \frac{\mySPY{y}{\tau}}{\myNormY{y}\myNormY{\tau}}
\end{align}
over a training set in $\mathcal{C}\times\mathcal{C}$ with $\myYmu[(\mu,\nu)] := \text{span} \{~\myybkRB[\mu],\myybkRB[\nu],\myYmuRB[\mu],\myYmuRB[\nu]~\}$ with RB approximations $\myybkRB[\mu]$ and $\myybkRB[\nu]$ of the best-knowledge states $\myybk[\mu]$ and $\myybk[\nu]$.
Note that this additionally requires $\myYRB$ to approximate the best-knowledge model \eqref{eq:bk}.
The spaces $\myYmuRB[(\mu,\nu)]$ can be computed efficiently online from the RB space $\myYRB$. It is therefore possible to consider different sets of training parameters in the course of the algorithm to reduce the required amount of memory.

\subsection{Construction of reduced basis spaces}
\label{sec:stepwise}

We next consider \blue{the} construction of the reduced basis spaces $\myYRB$ and $\myURB$. If the measurement space $\myT$ and the measurements --- and thus the data state $\mydata \in \myT$ --- are known {\it a priori}, we can directly use a greedy approach similar to the one for optimal control problems in~\cite{Kaercher_optCon}, see \cite{Masterarbeit_Nicole}. However, here we propose a different approach which does not require $\myT$ or $\mydata$ to be known {\it a priori}.

By varying in \eqref{prob:truth:sp} separately over the spaces $\myU$ and $\myY$, and inserting $\myd = \mydata - \myProj{\myT}\myy$, the following system is obtained for the model correction $\myu$, state $\myy$, and adjoint solution $\myw$:
\begin{subequations} \label{eq:t:pde}
\begin{align}
\mySPU{\myu}{\phi} - \myb(\phi,\myw) &= 0 && \forall ~\phi \in \myU \label{eq:t:pde:1}\\
\mya(\psi,\myw) - \lambda \mySPY{\psi}{\myd} &= 0 && \forall ~\psi \in \myY \label{eq:t:pde:2}\\
\mya(\myy,\psi) - \myb(\myu,\psi) &= \myfbk(\psi) && \forall ~ \psi \in \myY \label{eq:t:pde:3}\\
\mySPY{\myy + \myd}{\tau} &= \mySPY{\mydata}{\tau} && \forall ~ \tau \in \myT. \label{eq:t:pde:4}
\end{align}
\end{subequations}
We assume that $\myURB \subset \myU$ is fixed and low-dimensional, so that $\myu$ can be approximated well enough in $\myURB$ for the expected range of measurement data and the desired level of detail.
\blue{As $\myu$ is not known,} the RB space $\myYRB$ needs to provide good approximations \blue{to the solution of \eqref{eq:t:pde:3} for each model modification in $\myURB$ relative to its norm.
We realise this by constructing a state space $\myY_y \subset \myY$ as}
an RB approximation of the forward problem
\begin{equation}
\begin{aligned}
&\text{For } \mu \in \mathcal{C} \text{ and } f \in \{~\myfbk\} \cup \{~ \myb(\phi_{\rm{R,}m},\cdot):~1\le m\le M ~\} \subset \myY':\\
&\text{Find } y \in \myY \text{ s.t. } \mya(y,\psi) = f(\psi) \quad \forall ~ \psi \in \myY.
\end{aligned}
\end{equation}
If $\myT$ is not fixed by the experimental design, it can now be generated from $\myY_y$ with the greedy-OMP algorithm \ref{alg:greedyOMP}.
\blue{Given $\myT$ and f}ollowing an analogous argument as before for the generation of $\myY_y$, we can obtain an adjoint space $\myY_p \subset \myY$ if we iteratively replace $\myd$ in \eqref{eq:t:pde:2} with orthonormal basis functions $(\tau_l)_{l=1}^L$ of $\myT$ and perform an RB approximation of each equation over the parameter domain $\mathcal{C}$.
By considering a relative target accuracy, we can use $\lambda=1$ for the RB \blue{space generation, but note that since t}he approximation quality of the adjoint solution $\myw$ scales with $\lambda$, the target accuracy \blue{for} the RB \blue{adjoint} approximation should be chosen with respect to the largest regularisation parameter $\lambda$ that the RB 3D-VAR method is expected to be used for.
\blue{Finally, we set} $\myYRB:= \myY_y + \myY_p$.

\blue{T}he following sketch \blue{summarises this} consecutive construction of the spaces:
\begin{center}
\begin{minipage}[c]{0.95\textwidth}
\begin{tikzpicture}[node distance = 2.5cm, auto]
\node (l1) {$\myURB$};
\node [right of=l1](l2){$\mathcal{Y}_y$};
\node [right of=l2](l3){$\myT$};
\node [right of=l3](l4){$\mathcal{Y}_p$};
\node [right of=l4](l5){$\myYRB=\myY_y+\mathcal{Y}_p$.};
\path [line] (l1) -- node {\eqref{eq:t:pde:3}} (l2);
\path [line] (l2) -- node {\textit{Alg.} \ref{alg:greedyOMP}} (l3);
\path [line] (l3) -- node {\eqref{eq:t:pde:2}} (l4);
\path [line] (l4) -- (l5);
\end{tikzpicture}
\end{minipage}
\end{center}
\blue{T}his stepwise selection of the spaces avoids two possible drawbacks:
First, the data state $\mydata$ needs not to be known at the start of the offline phase;
the RB spaces are hence not influenced by measurement noise.
Second, once the spaces are constructed, they can be used repeatedly for different measurement data $\mydata$ without necessitating additional offline costs.

As $\myYRB= \myY_y + \myY_p$ comprises $M+L+1$ RB approximation spaces, $\myYRB$ may become large. The target accuracy and the training set should thus be chosen carefully to reduce the offline computational time.
Once the data state $\mydata$ becomes known, $\myYRB$ can be condensed by following the two-step RB approach in~\cite{EHK+2012}, i.e., by subsequently applying a greedy algorithm to derive spaces of smaller dimensions.


\section{Numerical Results}
\label{sec:experiments}
In this section we present numerical experiments to verify our theoretical results.
All computations were performed with MATLAB\textsuperscript{\textregistered} on a computer with 2.5 GHz Intel Core i5 processor and 4 GB of RAM.
%
%
%
\subsection{Model Description}
We consider the steady state temperature distribution $y$ of a thermal block $\Omega = (0,1)\times(0,1)$ that comprises up to three different material properties within
$\Omega_1 = (\text{\small $\frac{1}{4}$},\text{\small$\frac{3}{4}$}) \times (\text{\small$\frac{1}{4}$},\text{\small $\frac{3}{4}$})$,
$\Omega_2 = (0,1)\times(\text{\small$\frac{1}{2}$},1) \setminus \Omega_1$, and
$\Omega_{0} = (0,1)\times(0,\text{\small$\frac{1}{2}$}) \setminus \Omega_1 $.
We consider parameters $\mu = (\mu_1,\mu_2) \in \mathcal{C} := [\frac{1}{10},10]^2$, where $\mu_i$, $i \in \{1,2\}$ is the ratio of the thermal conductivity of $\Omega_i$ to $\Omega_0$.
We divide the outer boundary $\partial \Omega =  \myGin \cup \myGdir \cup \myGneu$, with $\myGneu = \{0,1\} \times [0,1]$, $\myGdir = [0,1]\times \{1\} $, and $\myGin = [0,1] \times \{0\}$,
and confer different boundary conditions on each by $\nabla y \cdot n = 0$ a.e. on $\myGneu$ (zero Neumann flux), $y|_{\myGdir} = 0$ a.e. on $\myGdir$ (zero Dirichlet) and $\nabla y \cdot n = u$ a.e. on $\myGin$ (Neumann flux).
Here, $n$ is the outer unit normal to $\Omega$ and $u$ is a yet unspecified function in $\myUcont := L^2(\myGin)$.
Across the subdomains, we  require that both the temperature $y$ as well as the heat flux is continuous.
An outline of this setup is provided in Figure \ref{fig:model}(a).

\begin{figure}
\qquad \qquad \quad
\subfloat[domain decomposition]
{\def\svgwidth{0.35\textwidth}
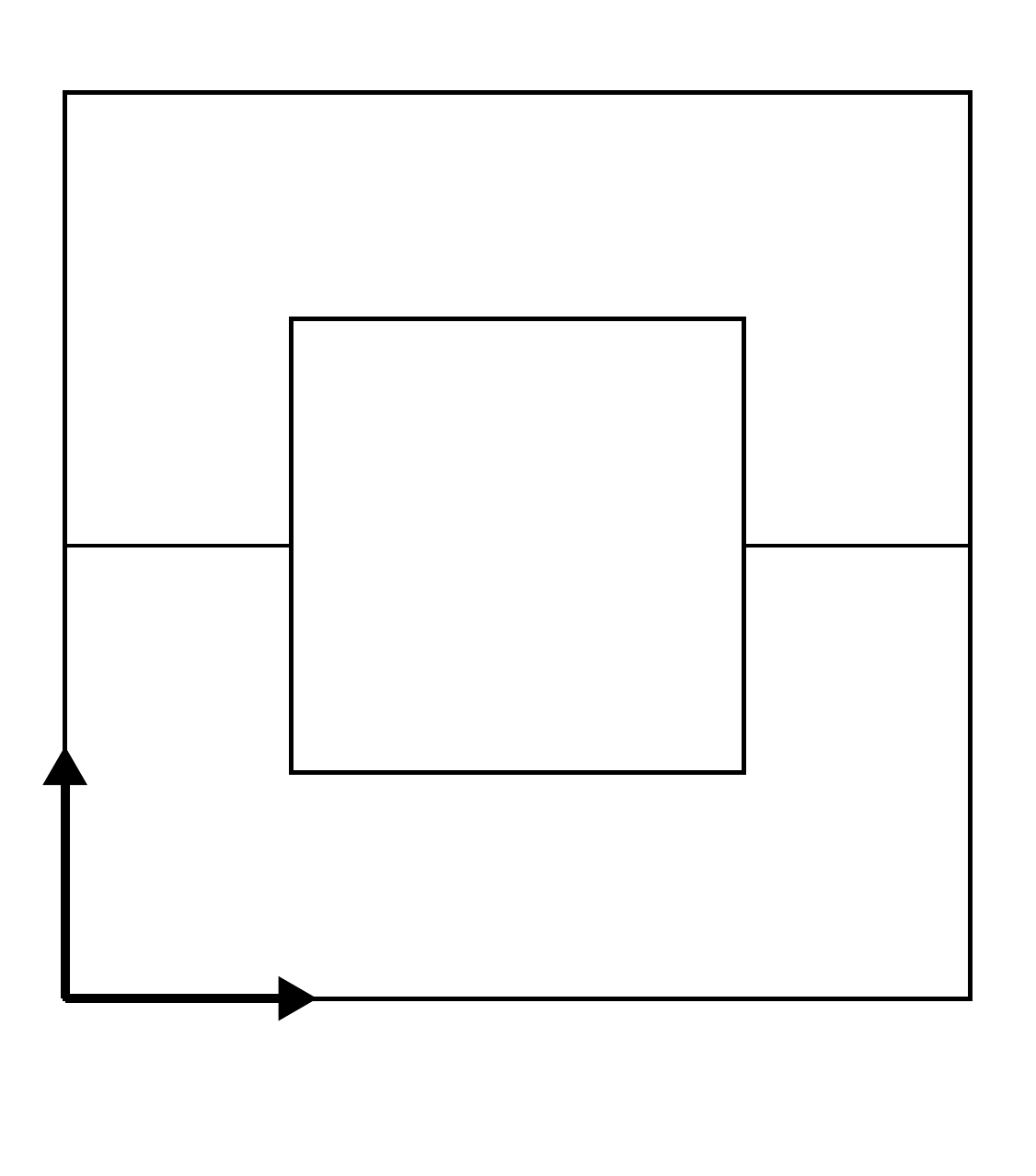}
\qquad \qquad \qquad
\subfloat[sensor placement]
{\includegraphics[width = 0.35 \textwidth]{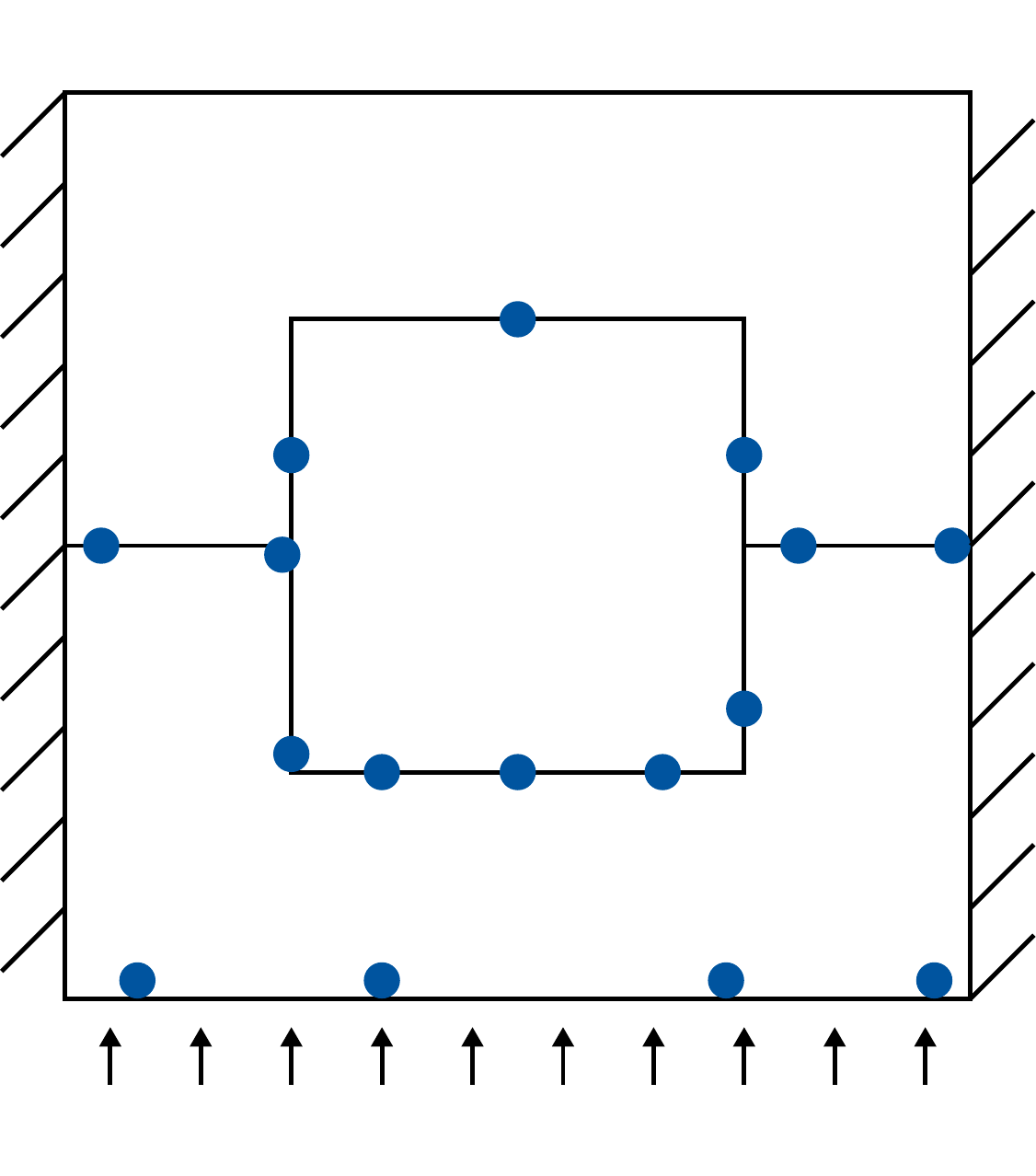} }
\caption{(a) Illustration of the domain and the imposed outer boundary conditions.
(b) Greedy placement of the measurement sensors.}
\label{fig:model}
\end{figure}

For the state space, we define $\myYcont := \{\, y \in \Hkp[1]:\, y|_{\myGdir} = 0\, \}$ with inner product $\mySP[\myYcont]{y}{w} := \mySPY{y}{w} := \int_{\Omega} \nabla y(\mathbf{x}) \cdot \nabla w(\mathbf{x}) d\mathbf{x}$.
For the space discretization in the state variable we use a linear finite element space $\myY$ of dimension 4,210.
The bilinear forms $\mya : \myYcont \times \myYcont \rightarrow \mathbb{R}$ and $b : \myUcont \times \myYcont \rightarrow \mathbb{R}$ associated with the thermal block problem are then given by
\begin{align*}
\mya(y,w) := \sum_{i=0}^2 \mu_i \int _{\Omega_i} \nabla y(\mathbf{x}) \cdot \nabla w(\mathbf{x})~d\mathbf{x}
\quad \quad \quad
b(u,w) := \int _{\myGin} u(\mathbf{x})~w(\mathbf{x})~dS(\mathbf{x}),
\end{align*}
where $\mu_0 := 1$.
Both forms are continuous, and $\mya$ is coercive with $\myalpha = \min \{1,\mu_1,\mu_2\} \ge 0.1$.

For the state and parameter estimation, we assume that the unknown true state $\myytrue \in \myY$ is given as the unique solution of the weak PDE
\begin{align}\label{eq:numExp:true}
\mya[\mymutrue](\myytrue, \psi) = b(\myutrue,\psi) \qquad \qquad \forall ~\psi \in \myY,
\end{align}
with $\myutrue(x_1):= 1.5+0.3\sin(2\pi x_1)$ for $x_1 \in \myGin$, and $\mymutrue = (7,0.3)$.
For the numerical computation of this state we approximate $\myutrue$ with 69 linear finite elements on $\myGin$.
For the best-knowledge model, we presume a steady, homogeneous heat inflow in the form of a uniform Neumann flux $\myustart \equiv 1$ on $\myGin$.
We hence obtain $\myfbk = b(\myustart,\cdot) \in \myY'$.
In this numerical experiment, we aim to estimate $\mymutrue$, $\myutrue$ and $\myytrue$ from a few measurements of $\myytrue$.

\subsection{Space Generation with Prior Knowledge} \label{sec:exp:generation}
In this section, we first specify the framework for the application of the 3D-VAR method.
We then use the approach described in the previous section to generate the RB spaces and the measurement space $\myT$.

Due to the diffusion of heat, state changes brought forth by local details in the Neumann flux smooth out as the distance to the boundary $\myGin$ increases.
The observation of such details would hence necessitate the placement of observation functionals very close to the boundary, and the local reconstruction of the Neumann flux would then be sensitive to noise.
For the model modifications, we hence make the educated guess that $\myutrue$ can be approximated sufficiently in the space $\mathbb{P}_3(\myGin)$ of polynomials on $\myGin$ with degree smaller or equal to 3; accordingly, we fix $\myU := \mathbb{P}_3(\myGin)$ equipped with the $L^2(\myGin)$ inner product.
As a consequence of this decision, we accept that we can approximate $\myutrue$ only up to the accuracy $\myNorm[L^2(\myGin)]{\myProj{\myU^{\perp}} \myutrue} \approx $1.9877e-02.
Similarly, for the true parameter $\mymutrue$, $\myytrue$ can only be approximated upto $\myNormY{\myProj{\myYmu[\mymutrue]^{\perp}}\myytrue} \approx$ 4.7576e-03.

We next generate the measurement space $\myT$ and the RB spaces $\myURB$ and $\myYRB$.
Since we already restricted $\myU$ to a small dimension, we chose $\myURB := \myU$ without discarding any further model modifications.
The state space $\myY_y \subset \myY$ is generated from the first four Legendre polynomials by a weak greedy algorithm over a $41\times 41$ regular training grid on the logarithmic parameter domain with target accuracy $10^{-5}$ relative to the norm of the state solution.
The algorithm terminated with $\dim \myY_y = 64$. We note that if we were to use this state space for a $\mu$-independent PBDW state estimate of $\myytrue$, we would need at least 64 measurement functionals for well-posedness, whereas any number of measurements is sufficient for well-posedness of the 3D-VAR method.

The measurement space $\myT$ was obtained from $\myY_y$ with the greedy-OMP algorithm \ref{alg:greedyOMP}, which selected 16 measurement functionals $g_l \in \myY'$, $l \in \{ 1,...,16\}$.
The library consists of gaussian functionals in $\myY'$ with standard deviation $0.01$ and centres in a $97\times 97$ regular grid on $(0.02,0.98)^2 \subset \Omega$.
We use the target value $\beta_0 := 0.5$ for the inf-sup constants $\mykappaRBmu$ and $\beta_{\myT,\rm{R}}(\mu,\nu)$
with parameters $\mu$, $\nu$ in a $21 \times 21$ grid on $\mathcal{C}$.
Figure \ref{fig:spaces}(a) shows the development of the minimal inf-sup constants $\mykappaRBmu$ 
and $\mykappaRBmu[\mu,\nu]$
over the respective training set as the greedy-OMP Algorithm \ref{alg:greedyOMP} expands $\myT$.
Due to the margin at the domain boundary for the placement of the sensors, we cannot expect the inf-sup constant to reach 1 asymptotically.
Since $\min_{\mu} \mykappaRBmu$ exceeds the target value already for five measurement functionals, most measurements have been chosen to increase $\mykappaRBmu[\mu,\nu]$ over $\mathcal{C}^2$.
The centres of the chosen measurement functionals are indicated in Figure \ref{fig:model}(b).
The four measurements near $\myGin$ are most important for determining the optimal model correction, whereas the other measurements are most important for estimating unknown parameters.
In Figure \ref{fig:spaces}(b), the $\myHeins^0(\mu)$-coercivity constant $\alpha_A^0(\mymutrue,\lambda)$ of $A$ and the corresponding lower bound $\mydeltaLB$ from Theorem \ref{thm:tCoercivity} are plotted over $\lambda$ for the high-fidelity spaces.
We observe that the lower bound closely tracks the behaviour of $\mydelta$.
Asymptotically, $\mydelta$ is larger than $\mydeltaLB$ by a factor of 1.0325.

\begin{figure}
\centering
\subfloat[Greedy-OMP Algorithm]
{\includegraphics[width = .45\textwidth]{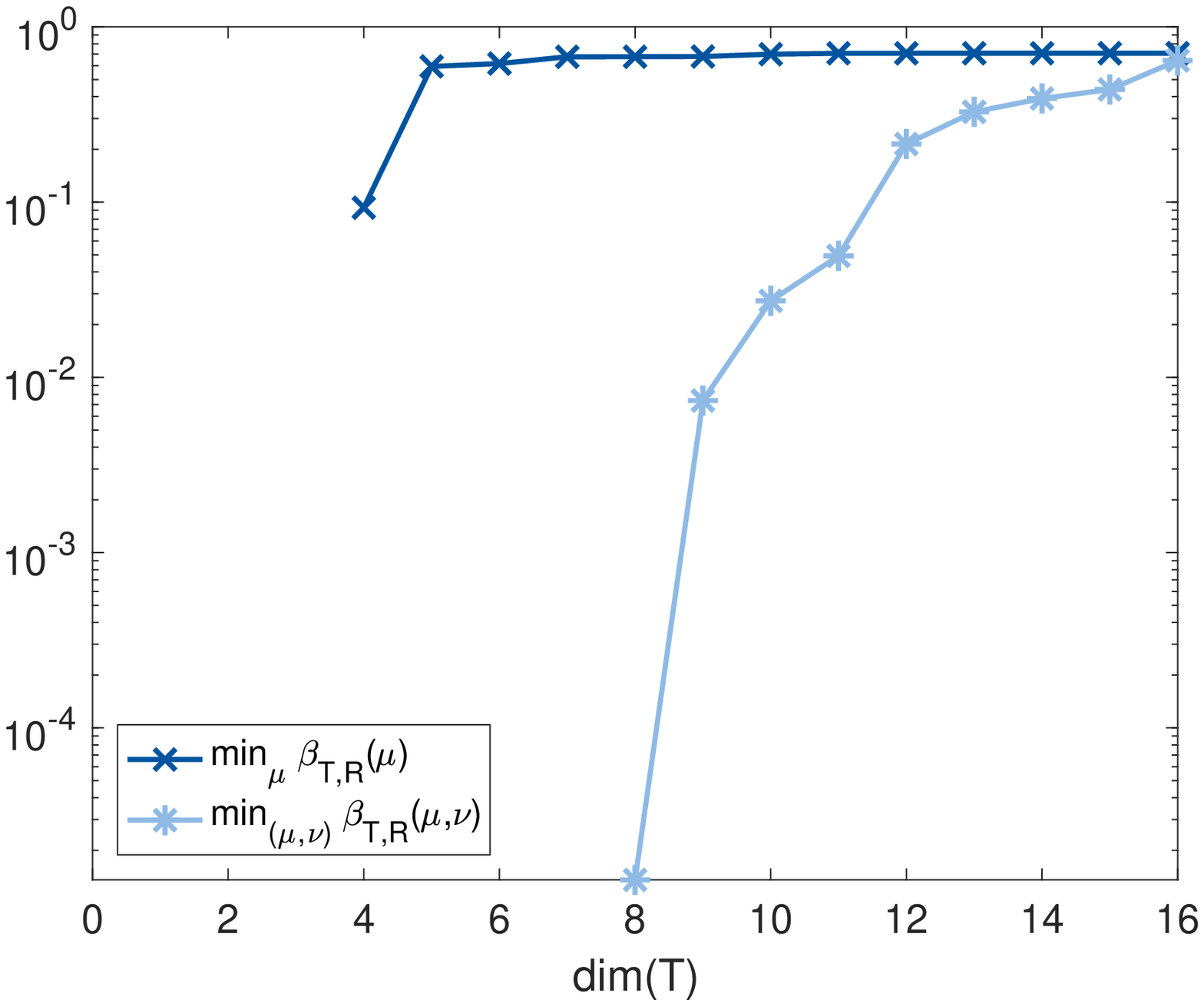} }
~~
\subfloat[$\myHeins^0(\mu)$-coercivity]
{\includegraphics[width = .45\textwidth]{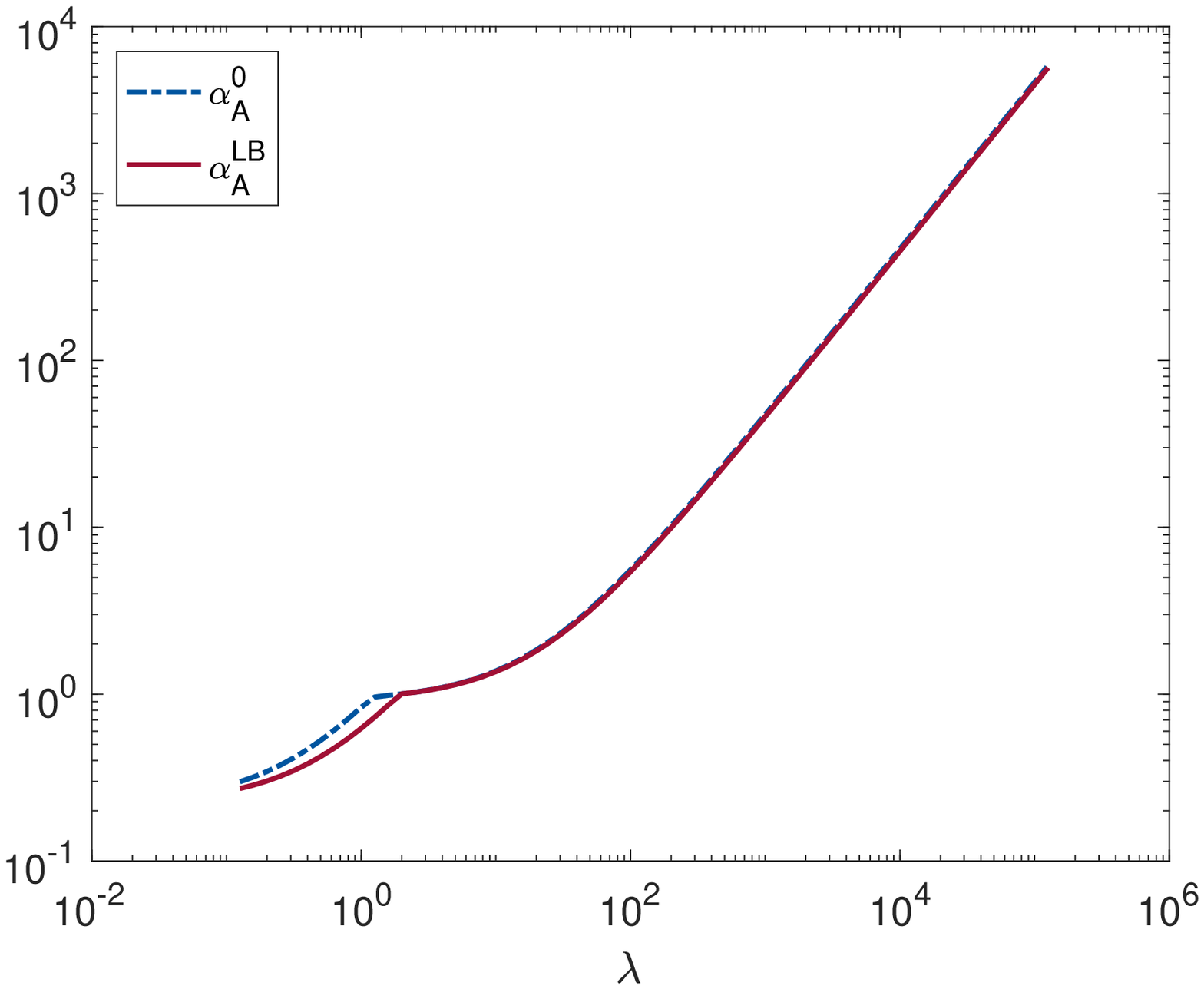} }
\caption{(a) Development of the target variables of the greedy-OMP algorithm versus the dimension of $\myT$.
(b): The $\myHeins^0(\mu)$-coercivity constant $\alpha_A^0(\mu, \lambda)$ versus $\lambda$ and its lower bound $\mydeltaLB(\lambda, \myetainf, \myetasup, \mykappa)$ for $\mu = \mymutrue$.}\label{fig:spaces}
\end{figure}

To finish the generation of $\myYRB$, we expand the state space $\myY_y$ with an adjoint space $\myY_p$ with $\dim \myY_p = 95$.
This space is computed from an orthonormal basis of $\myT$ with an RB approximation of the adjoint equation \eqref{eq:t:pde:2}.
Since prior tests indicate that due to the local influence of the measurement functionals the weak greedy algorithm only chooses training parameters very close to the boundary $\partial \mathcal{C}$, the computation of $\myY_p$ was done using a relative target accuracy of $10^{-5}$ on 40 regularly spaced training parameters on $\partial \mathcal{C}$ in the logarithmic plane.
After the computation of $\myY_p$, the target accuracy was confirmed on a fine test grid over the whole parameter domain.
Altogether, the offline phase for the RB space generation, including the selection of appropriate measurement functionals, finished after 463 seconds with $\dim \myURB = 4$, $\dim \myT = 16$ and $\dim \myYRB = 159$.

\subsection{Reduced Basis Approximation}
In this section, we evaluate 1) the computational efficiency of the RB solution, 2) its accuracy and the effectivity of the a posteriori error bounds, and 3) the effects of noise on the approximation of $\myutrue$.
To this end, we take measurements $\mathbf{m}_{\rm{d}} = (g_l(\myytrue))_{l=1}^{16} \in \mathbb{R}^{16}$ of the true state to compute $\mydata = \myProj{\myT}\myytrue$.
Noisy measurements are constructed by adding a random variable drawn from the probability distribution $\mathcal{N}(0,0.01^2)$ upon each measurement before the computation of $\mydata$.
This corresponds to a noise level of approximately 1.5\% in each measurement compared to the difference to the measurements of the best-knowledge state $\myybk[\mymutrue]$.

First, we evaluate the \blue{computational efficiency} of the RB method \blue{in comparison to} the truth 3D-VAR method:
For 200 random parameters in $\mathcal{C}$ and $\lambda = 100$, the truth 3D-VAR solution was computed from different noisy measurements.
The mean computation time was 7.08 $s$.
The RB 3D-VAR solution was then computed over the same parameters with the same noisy data for comparison.
With an average online computation time of 4.2 $ms$ for the computation of the RB 3D-VAR solution and 1.3 $ms$ for the computation of the a posteriori error bounds, the RB method showed a mean speedup of 1,276.

\begin{figure}
\centering
\subfloat[model correction]
{\includegraphics[width = 0.45\textwidth]{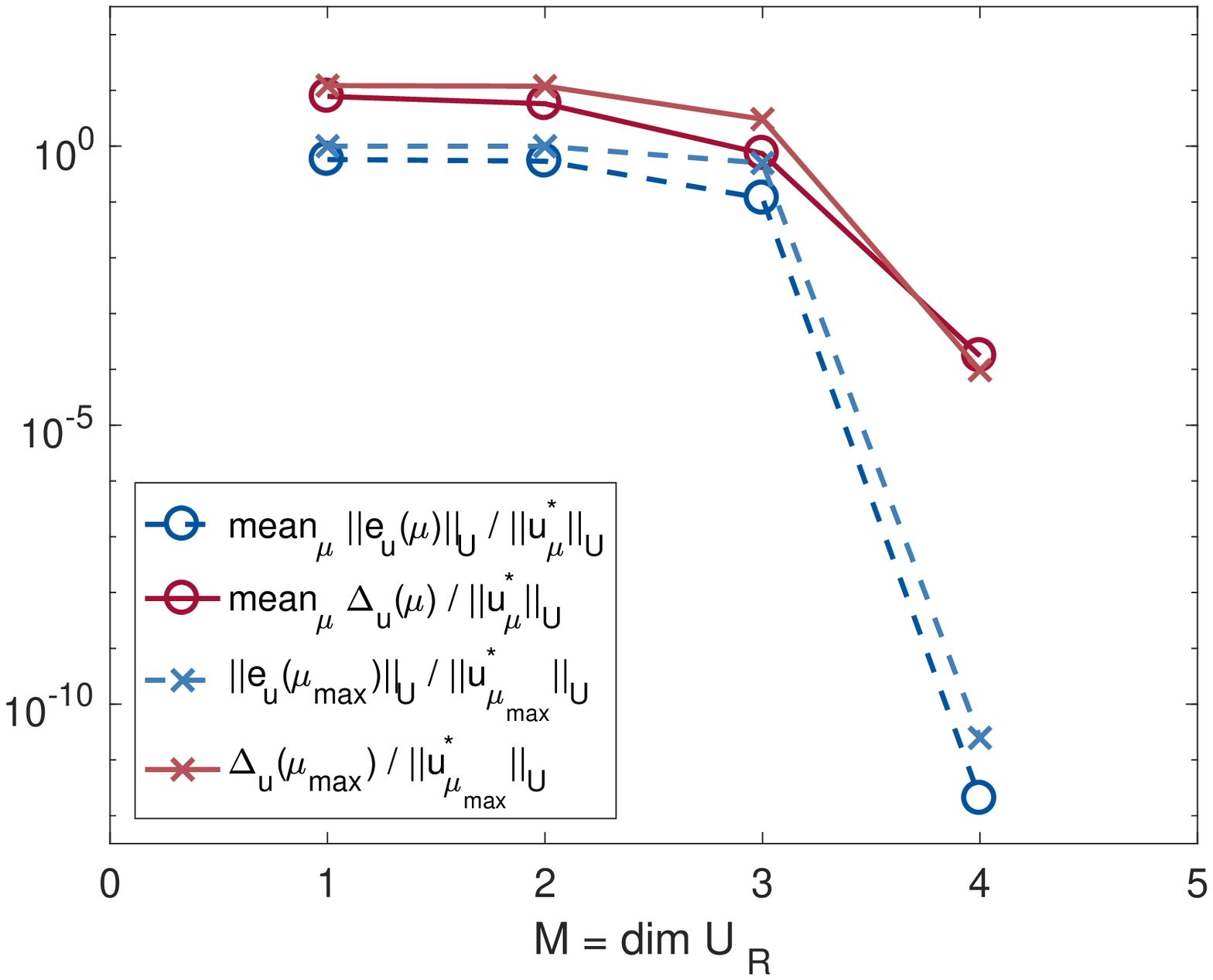} }
~~
\subfloat[state]
{\includegraphics[width = 0.45\textwidth]{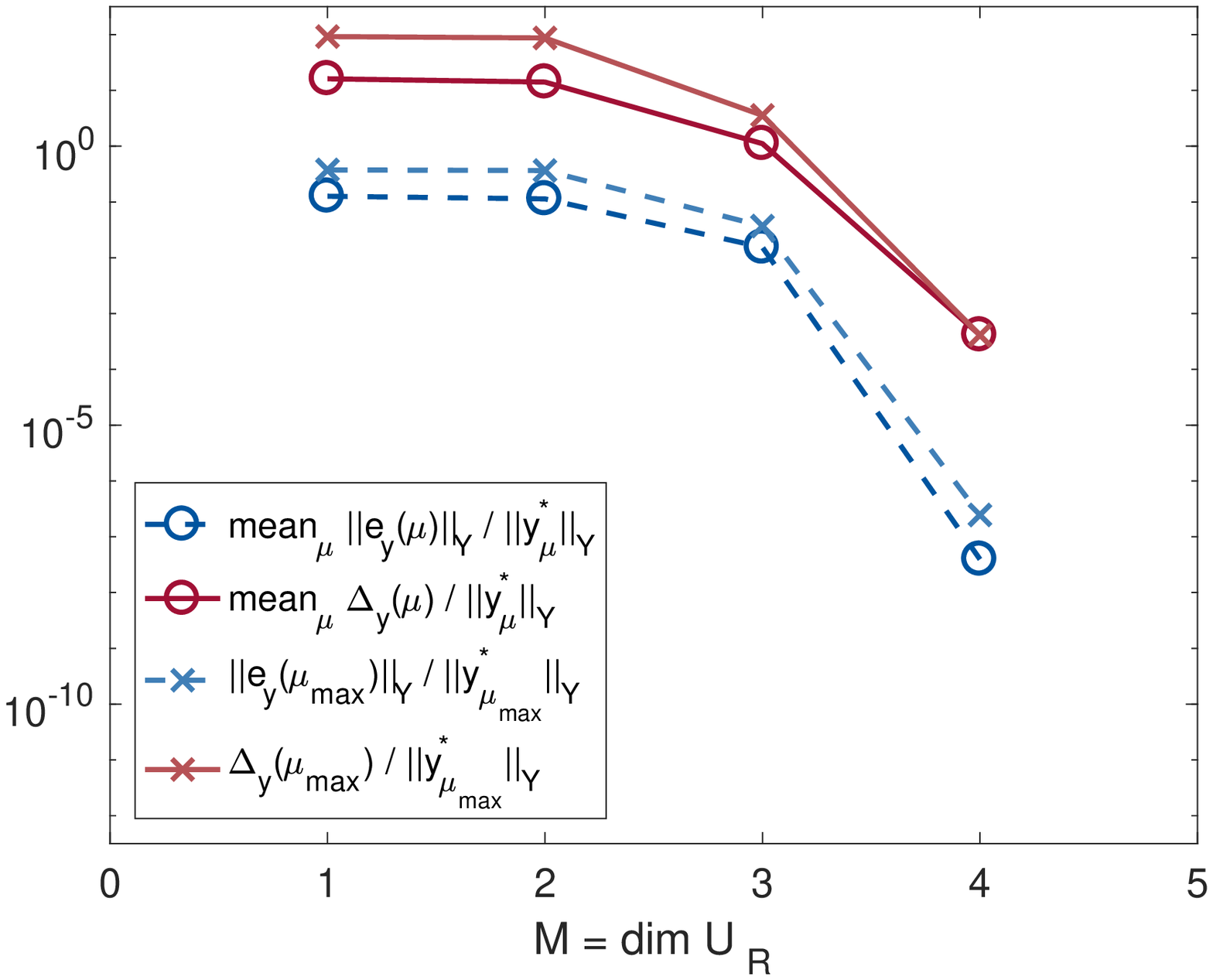} }

\subfloat[observable misfit]
{\includegraphics[width = 0.45\textwidth]{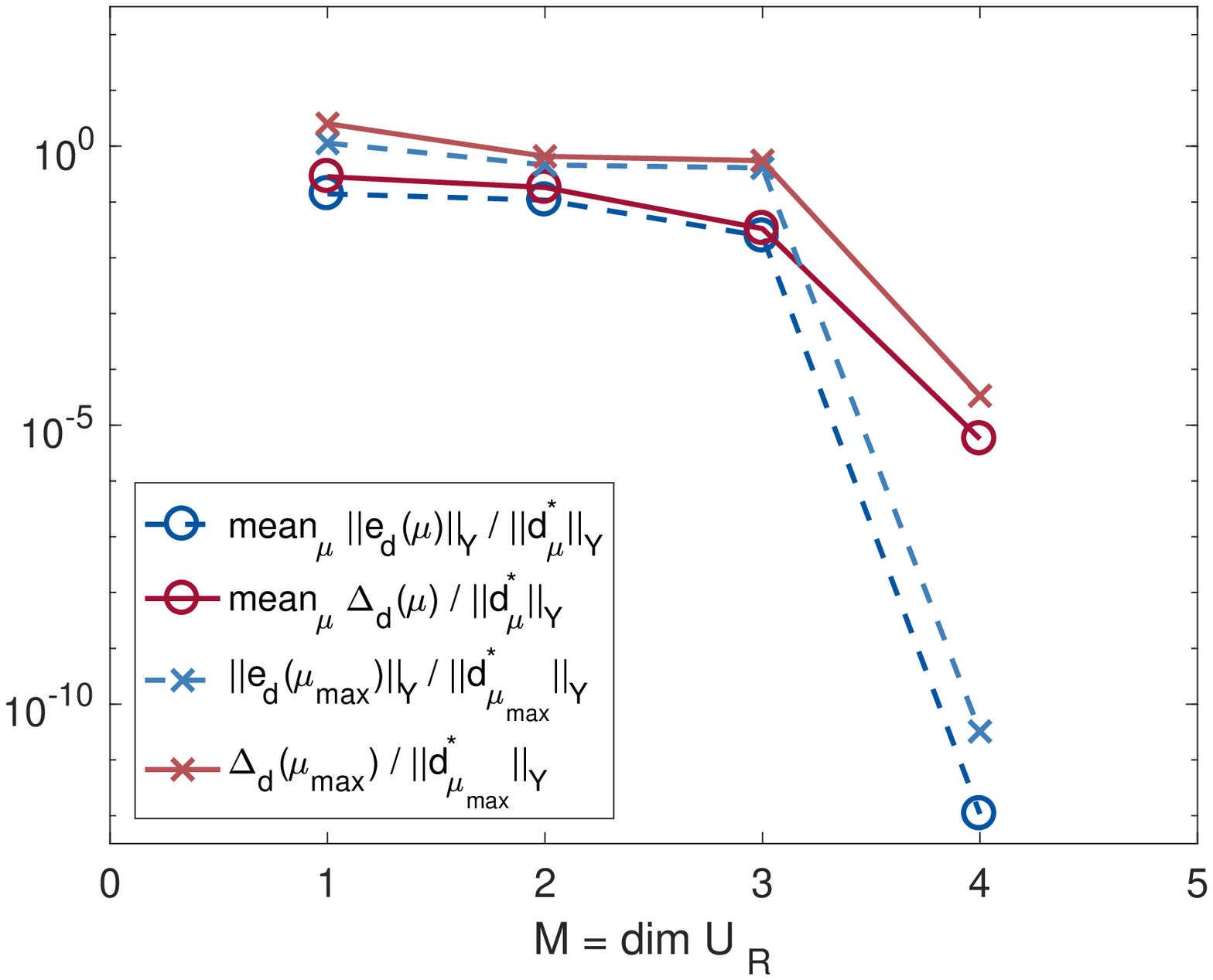} }
~~
\subfloat[adjoint]
{\includegraphics[width = 0.45\textwidth]{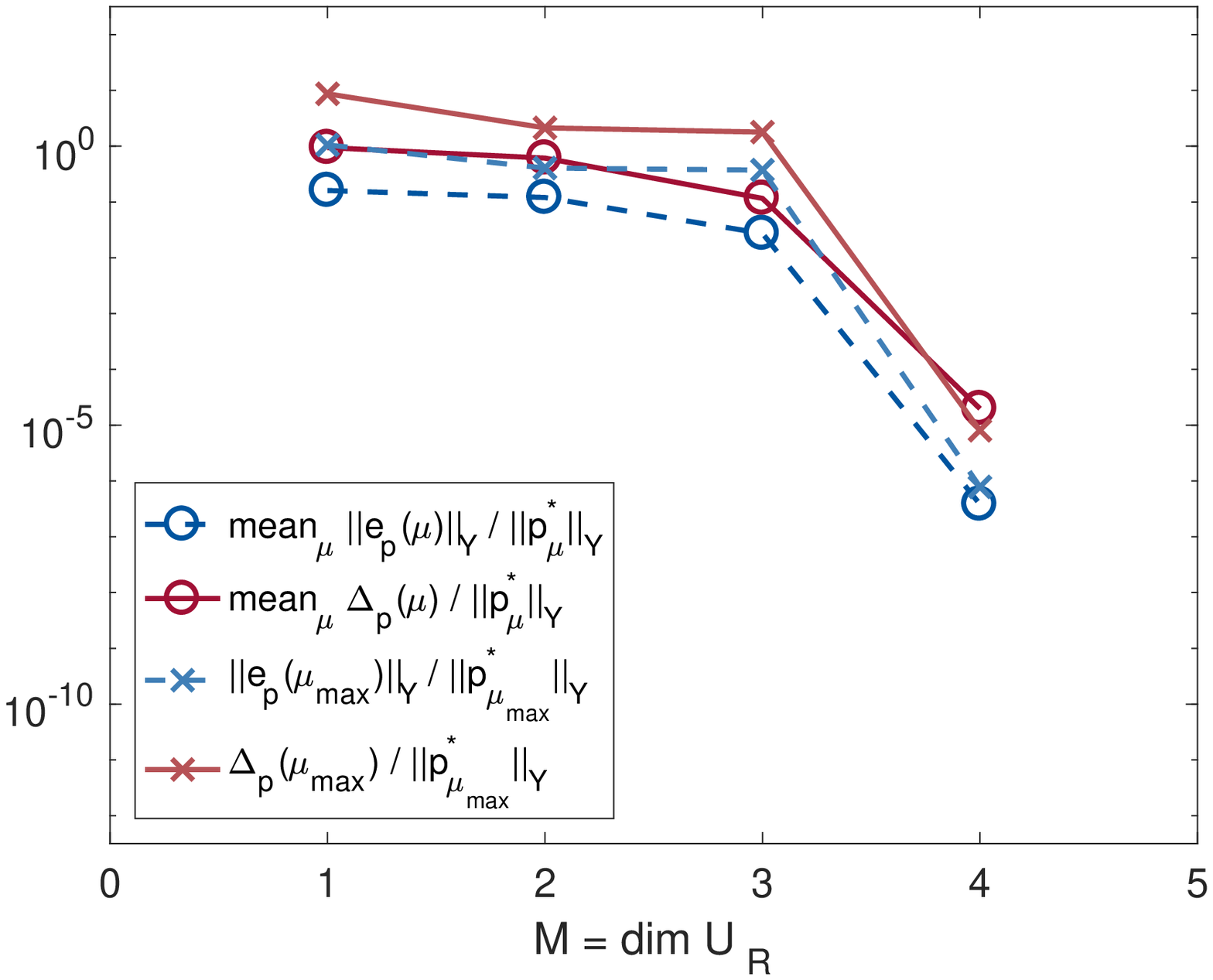} }

\caption{
Mean relative error \blue{(blue, dashed, o-marks)} and relative error bound \blue{(blue, dashed, x-mark)} in the RB approximation for $\lambda = 100$ from noisy data for the model correction, state, observable misfit, and adjoint for the stepwise approach with different $\myURB = \mathbb{P}_j(\myGin)$, $j=1,...,4$.
\blue{
In each plot the red, continuous, o-marked curve marks the maximum of the relative truth error (labelled shorthand as evaluation at $\mu_{\max}$) for each $j$, whereas the red, continuous, x-marked curve shows the corresponding relative error bound evaluated at the same parameter and the same $j$.
}
The mean and maximum are computed over 200 random parameters.
}
\label{fig:exp:effectivity}
\end{figure}

To assess the effectivity of the a posteriori error bounds, we generate additional pairs $(\myURB^j,\myYRB^j) \subset \myU \times \myY$, $j \in \{1,2,3\}$ of RB spaces with $\myURB^j = \mathbb{P}_{j-1}(\myGin)$ the span of polynomials on $\myGin$ of degree smaller or equal to $j-1$.
We kept the measurement space $\myT$ fixed, but otherwise followed the approach as outlined in section \ref{sec:exp:generation} to obtain $\myYRB^j$.
For the same parameters and the same noisy data as before, we evaluate the error between the truth 3D-VAR solution and the RB solution on these additional spaces.
Figure \ref{fig:exp:effectivity} shows the mean error and mean a posteriori error bound relative to the norm of the solution variable versus the dimension $j = \dim \myURB^j$ of the RB model correction space.
Additionally, the maximum relative error and the corresponding relative error bound are indicated.
As the polynomial $x^3$ cannot be approximated in $\myURB^j$ for $j \le 3$, the large portion of $x^3$ in the truth 3D-VAR model modification $\myu$ leads to the strong decrease in the error for $\dim \myURB = 4$.
We observe that the effectivity of the error bounds is generally small and constant, except for the case $\myURB = \myU$ where $\myeu$ and $\myed$ lie below the natural limits of the error bounds $\myDu$ and $\myDd$ dictated by the square root of the machine accuracy.
For different regularisation parameters $\lambda \le 10^3$, we observe a posteriori error bounds of a similar magnitude.

\begin{figure}
\centering
\subfloat[$\lambda = 1$]
{\includegraphics[width = 0.45\textwidth]{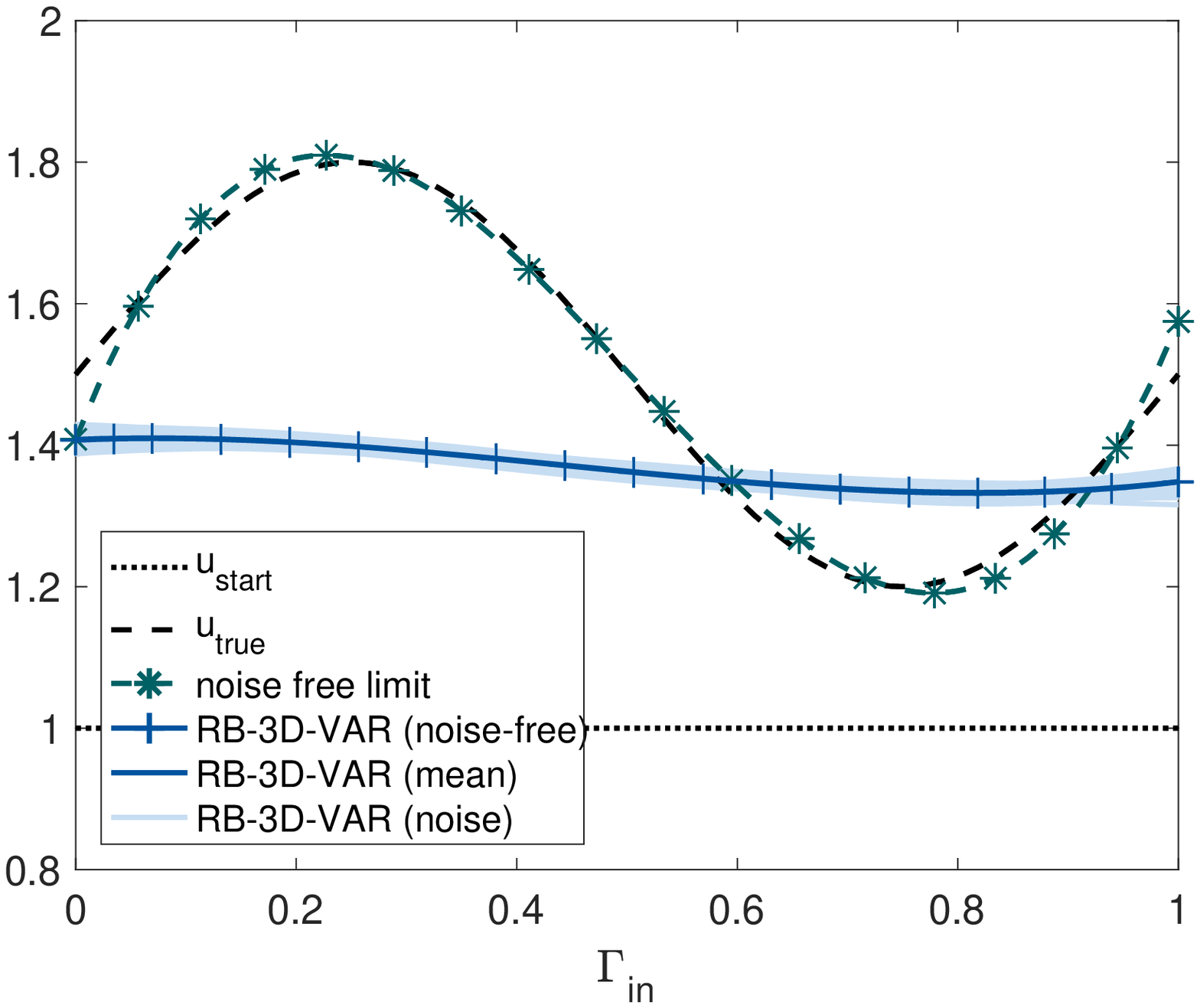} }
~~~
\subfloat[$\lambda = 10$]
{\includegraphics[width = 0.45\textwidth]{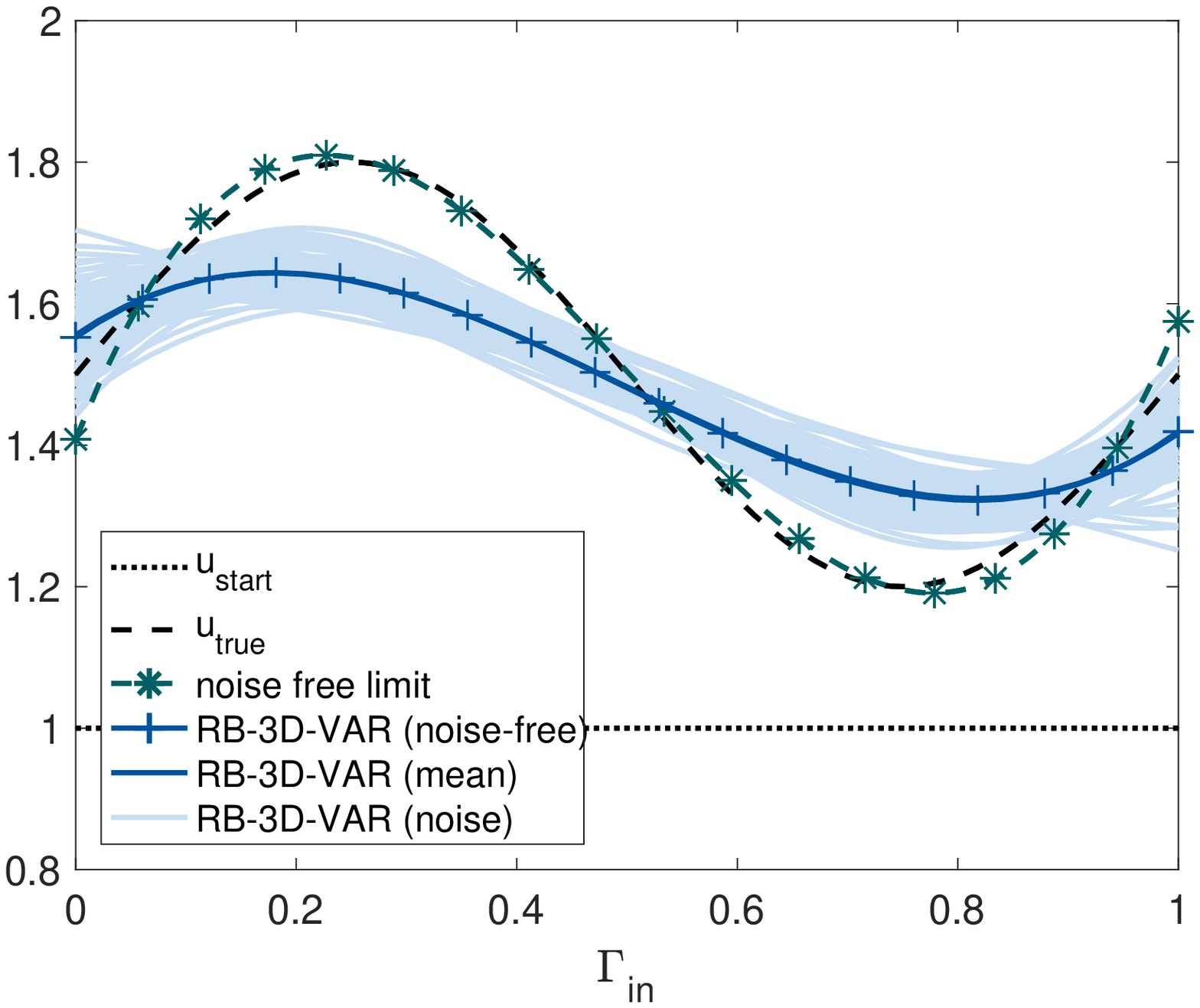} }

\subfloat[$\lambda = 100$]
{\includegraphics[width = 0.45\textwidth]{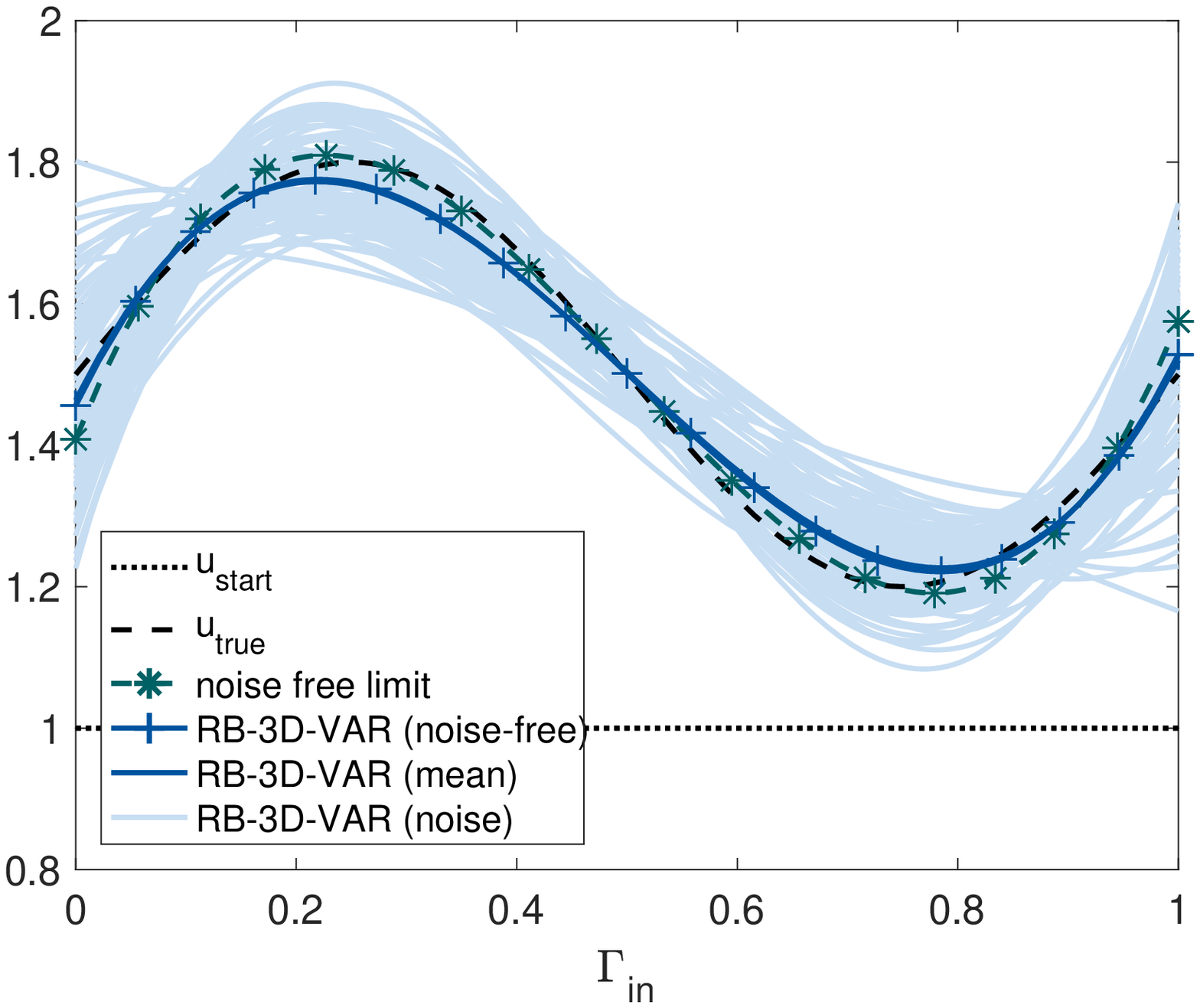} }
~~~
\subfloat[$\lambda = 1000$]
{\includegraphics[width = 0.45\textwidth]{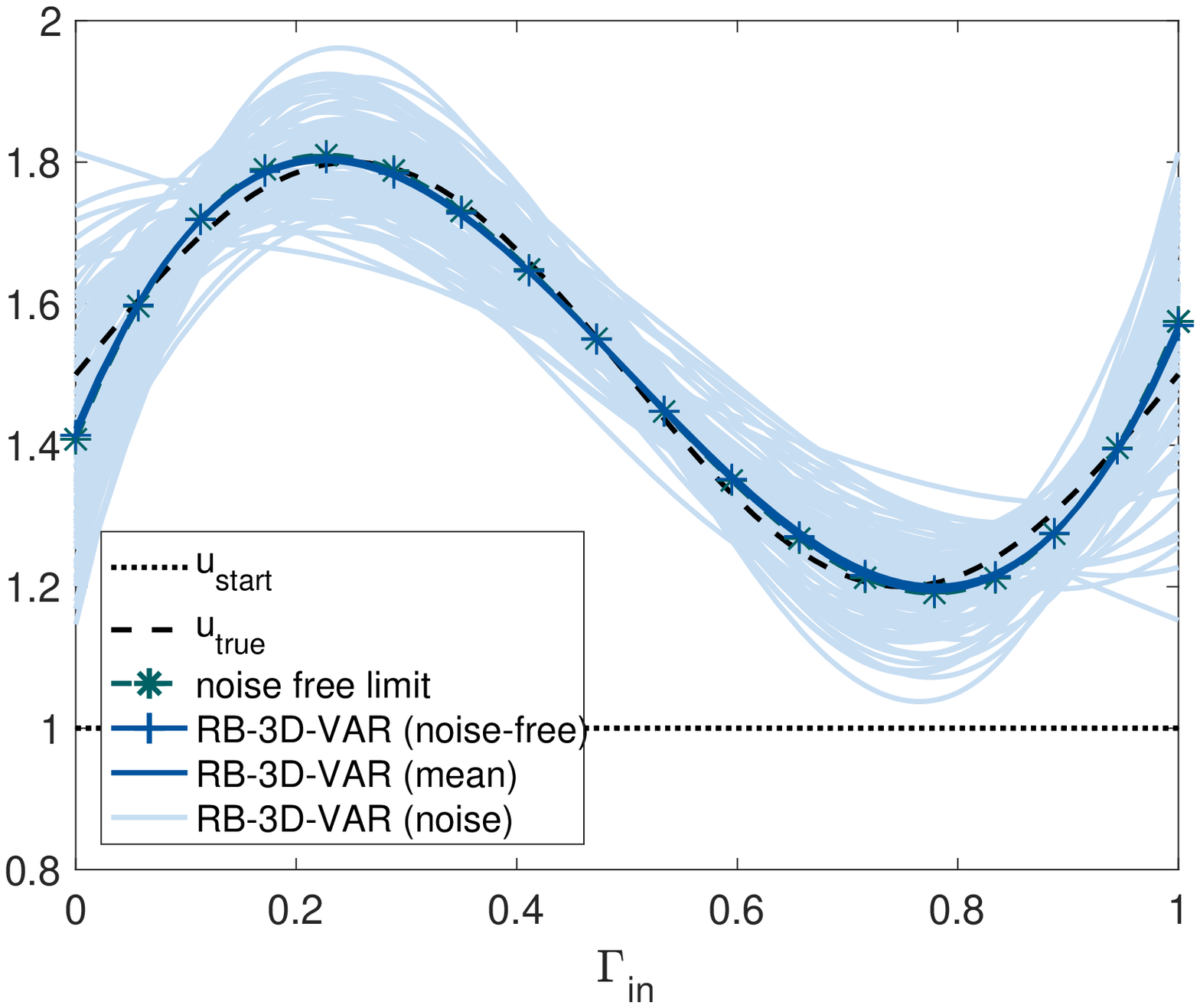} }
\caption{Qualitative behaviour of the 3D-VAR solution as an approximation of $\myutrue$ for the parameter $\mymutrue$.}
\label{fig:utrue:approx}
\end{figure}

Given that the RB model correction $\myuRB$ approximates $\myu$ so accurately, we use the RB method to investigate how well the qualitative, global behaviour of the true Neumann flux $\myutrue$ can be reflected by the 3D-VAR method, if the true parameter $\mymutrue$ is provided.
We hence compute $\myuRB[\mymutrue]$ for 100 different noisy data sets and for regularisation parameters $\lambda = 10^i$, $i=0,...,3$.
The results are shown in Figure \ref{fig:utrue:approx}.
We observe that for $\lambda = 1$, we obtain a good compromise between $\myutrue$ and our initial guess $\myustart \equiv 1$ in the best-knowledge model, and the measurement error only leads to small deviations.
As we increase $\lambda$ in favour of the biased measurements, the estimates start to differ, but, as can be seen when comparing \blue{$\lambda = 10^2$} to \blue{$\lambda = 10^3$}, the difference to the noise-free solution is bounded.

\subsection{Parameter Estimation}
\label{sec:paraEst}
We finally test three different approaches for parameter estimation with the 3D-VAR method to obtain approximations for $\mymutrue = (7,0.3)$, $\myutrue$ and $\myytrue$; to assess the estimation quality in general, we first use noise-free data.
As the involved minimization process takes a long time to converge, we employ the RB 3D-VAR method for comparison.
Finally, we evaluate the sensitivity of the parameter estimate with respect to noise in the data.

We define three different cost functionals $J_i^{\lambda} : \mathcal{C} \rightarrow \mathbb{R}$ by
\vspace{-1ex}
\begin{equation} \label{eq:J:truth}
\begin{aligned}
\myJeins(\mu) &:= \frac12 \myNormU{\myu}^2 + \frac{\lambda}{2} \myNormY{\myd}^2
\qquad \qquad
\myJdrei(\mu) := \frac12 \myNormY{\myd}^2 \\
\myJzwei(\mu) &:= \frac12 \myNorm[2]{\mathbf{v}_{\rm{d}}-(g_l(\myy))_{l=1}^{16}}^2
\end{aligned}
\end{equation}
where $(\myu,\myy,\myw)\in \myU \times \myY \times \myY$ is the 3D-VAR truth solution \eqref{prob:truth:sp} and $\myd = \mydata-\myProj{\myT}\myy$ is the observable misfit.
We then approximate $\mymutrue$ with
\begin{align}
\mymui ~\in~ \text{arg} ~ \min _{\mu \in \mathcal{C}} \myJi (\mu), \qquad i \in \{1,2,3\}. \label{eq:min:J}
\end{align}
We thus obtain a bilevel optimization problem whose inner optimization requires the solution of the 3D-Var problem
For $i=2$, \eqref{eq:min:J} finds the parameter $\mymuzwei \in \mathcal{C}$, for which the measurements $(g_l(\myy[\mymuzwei]))_{l=1}^{16}$ are closest to the actual measurement values $\mathbf{v}_{\rm{d}}$.
In contrast, for $i=3$ the state $\myy[\mymudrei]$ is closest to $\mydata + \myT ^{\perp}$ over all 3D-VAR state solutions.
Due to our prior choice for $\myU$ with $\myutrue \notin \myU$, the measurements are not necessarily obtainable and we cannot expect $\mymui$ to converge to $\mymutrue$ for $\lambda \rightarrow \infty$.

We first consider the truth parameter estimation problem. The problems in \eqref{eq:min:J} for $i \in \{1,2,3\}$ were solved within $\mathcal{C}$ from the starting point $(\mu_1,\mu_2)=(1,1)$ using the matlab function \texttt{fminsearch}\blue{, which uses the gradient-free simplex search algorithm described in \cite{simplex}}.
We chose the target accuracy of 1e-12 in both the minimum value and the minimizing parameter.
Each minimization took between 25 and 28 minutes.
Table \ref{table:muapprox:truth} shows the results obtained for $\lambda = 10^j$, $j=0,...,3$.
Generally, the parameter estimate improves as $\lambda$ increases and the 3D-VAR solution favours closeness to the measurement data, with $i=3$ providing the best parameter estimate.
However, its accuracy appears to stagnate for $\lambda > 10^2$, presumably because of $\myutrue \notin \myU$.
The other two problems yielded comparative results for $\lambda = 1$, but did not improve as fast as the third for larger $\lambda$.

Concerning $\myutrue$ and $\myytrue$, we observe that the approximation through the 3D-VAR solution at the parameter estimate $\mymui$ improves with increasing $\lambda$ and comes close to the best-fit errors $\myNorm[L^2(\myGin)]{\myProj{\myU^{\perp}} \myutrue} \approx $1.9877e-02 and $\myNormY{\myProj{\myYmu[\mymutrue]^{\perp}}} \approx$ 4.7576e-03, which resulted from limiting $\myU$ to $\mathbb{P}_3(\myGin)$.
Note that even though the parameter estimate for $i=3$ becomes slightly worse when changing from $\lambda = 10^2$ to $10^3$, the approximation of $\myutrue$ and $\myytrue$ still improves.

\begin{table}
\caption{
The parameter estimates obtained from solving the minimizations \eqref{eq:min:J} for different $\lambda$, the logarithmic distance to the true parameter $\mymutrue$, the number of function evaluations, and the approximation quality of $\myutrue$ and $\myytrue$ for $\mu = \mymui$ through $u_{\mu}:=\myustart+\myu$ and $\myy$.
}
\label{table:muapprox:truth}       
\begin{tabular}{cc|cccccc}
\hline\noalign{\smallskip}
$i$ & $\lambda$  & $(\mymui)_1$ & $(\mymui)_2$& log. dist. & eval. & $\myNormU{\myutrue - u_{\mymui}}$ & $\myNormY{\myytrue-\myy[\mymui]}$ \\
\noalign{\smallskip}\hline\noalign{\smallskip}
1 		& $10^0$ & 4.7545 & 0.2173 & 2.1878e-01 &  235 & 4.5524e-01 & 2.3916e-01 \\
  		& $10^1$ & 6.0571 & 0.2666 & 8.1061e-02 &  224 & 1.9754e-01 &   9.8988e-02 \\
  		& $10^2$ & 6.8618 & 0.2952 & 1.1115e-02 &  219 & 3.5157e-02 & 1.5374e-02   \\
  		& $10^3$ & 6.9896 & 0.2997 & 8.1349e-04 &  226 &2.1062e-02  &  5.0379e-03  \\
\noalign{\smallskip}\hline\noalign{\smallskip}
2  		& $10^0$ & 3.2061 & 0.2370 & 3.5420e-01 &  228 &4.0448e-01  & 2.7434e-01   \\
	  	& $10^1$ & 5.6337 & 0.2928 & 9.4896e-02 &  225 &1.1741e-01  &  7.9138e-02  \\
  		& $10^2$ & 6.8335 & 0.2997 & 1.0461e-02 &  243 & 2.6032e-02 &  1.1198e-02  \\
  		& $10^3$ & 6.9876 & 0.3001 & 7.9396e-04 &  238 & 2.1109e-02 & 5.0631e-03   \\
\noalign{\smallskip}\hline\noalign{\smallskip}
3  		& $10^0$ & 5.6402 & 0.2513 & 1.2134e-01 &  226 &3.4210e-01  & 1.9846e-01   \\
 		& $10^1$ & 6.9380 & 0.2980 & 4.8032e-03 &  232 & 1.0528e-01 & 5.2142e-02   \\
  		& $10^2$ & 7.0038 & 0.3002 & 3.3277e-04 &  235 & 2.5843e-02 &  9.1612e-03  \\
  		& $10^3$ & 7.0047 & 0.3002 & 3.9092e-04 &  228 & 2.1138e-02 & 5.0527e-03   \\
\noalign{\smallskip}\hline
\end{tabular}
\end{table}

To speed up the parameter estimation, we replace the truth 3D-Var problem with its RB approximation and consider the ``reduced'' cost functionals
\begin{equation}\label{eq:J:RB}
\begin{aligned}
\myJeinsRB(\mu) &:= \frac12 \myNormU{\myuRB(\mu)}^2 + \frac{\lambda}{2} \myNormY{\mydRB(\mu)}^2
\quad \qquad \myJdreiRB(\mu) := \frac12 \myNormY{\mydRB(\mu)}^2 \\
\myJzweiRB(\mu) &:= \frac12 \myNorm[2]{\mathbf{v}_{data}-(g_l(\myyRB(\mu)))_{l=1}^{16}}^2 \\
\end{aligned}
\end{equation}
where for $\mu \in \mathcal{C}$ we have $(\myuRB,\myyRB,\mywRB)\in \myURB \times \myYRB \times \myYRB$ is the RB 3D-VAR solution and $\mydRB = \mydata-\myProj{\myT}\myyRB$ is the misfit.
We then take
\begin{align}
\mymuiRB &~\in~ \text{arg} ~ \min _{\mu \in \mathcal{C}} \myJiRB (\mu) && i \in \{1,2,3\}, \label{eq:min:JRB}
\end{align}
as RB parameter estimate.
The distance between $\mymuiRB$ and $\mymui$ using noise-free data, as well as the computational time for the RB parameter estimate and the corresponding speedup compared to using \eqref{eq:J:truth} are listed in Table \ref{table:muapprox:noise}(a).
The approximation of $\mymui$ through $\mymuiRB$ was very precise with a maximal logarithmic distance of 3.3e-08 in the parameter domain.
In each case, the minimization finished within 1.1 s, \blue{resulting in} speedup-factors between 1,498 and 1,680 compared to \blue{the truth evaluation} \eqref{eq:J:truth} \blue{requiring 25 - 28 min per parameter estimate}.
\blue{
The speedup thus justifies the initial offline cost of 7.7 min for the RB space generation, especially when considering that the RB spaces need only be generated once and can then be used ({\cal i}) for all combinations of $\lambda$ and $i$ in \eqref{eq:min:JRB} (and also other cost functions), thereby allowing for conclusions to be drawn from the comparison of multiple parameter estimates; and ({\cal ii}) to repeat the parameter estimation repeatedly for different measurement data.
}

\begin{table}
\caption{
(a) Parameter estimation with the RB 3D-VAR method for noise-free data: Distance to the truth parameter estimate $\mymui$ in the logarithmic parameter plane, the computational time, and the speedup compared to \eqref{eq:J:truth}.
(b) Minimum, mean and maximum distance on the logarithmic parameter plane between RB parameter estimates obtained from noisy data, and the unbiased RB parameter estimate.}
\label{table:muapprox:noise}       
\begin{tabular}{cc|ccc|ccc}
\hline\noalign{\smallskip}
~ & ~ & \multicolumn{3}{c}{(a)} & \multicolumn{3}{c}{(b)} \\
~ & ~ & \multicolumn{3}{c}{noise-free data} & \multicolumn{3}{c}{noisy data (log. dist. to noise-free $\mymuiRB$)} \\
$i$ & $\lambda$  & dist. to $\mymui$ & time [s] & speedup & min & mean & max \\
\noalign{\smallskip}\hline\noalign{\smallskip}
1 		& $10^0$ & 3.2603e-08 & 1.1017 & 1498.8 & 9.5939e-04 & 2.1314e-02 & 8.1847e-02 \\
  		& $10^1$ & 4.6456e-09 & 1.0347 & 1527.1 & 1.8072e-03 & 2.4407e-02 & 9.0962e-02 \\
  		& $10^2$ & 7.4937e-09 & 1.0147 & 1489.7 & 2.0184e-03 & 2.6270e-02 & 9.5310e-02 \\
  		& $10^3$ & 6.7160e-10 & 1.0178 & 1576.8 & 1.7981e-03 & 2.6575e-02 & 9.5909e-02 \\
\noalign{\smallskip}\hline\noalign{\smallskip}
2  		& $10^0$ & 4.5371e-09 & 0.9358 & 1679.5 & 7.2490e-04 & 1.3379e-02 & 4.4470e-02 \\
	  	& $10^1$ & 1.4731e-08 & 0.9443 & 1662.7 & 1.6227e-03 & 2.0544e-02 & 6.9932e-02 \\
  		& $10^2$ & 5.5204e-09 & 1.0675 & 1608.7 & 2.7511e-03 & 2.4773e-02 & 8.4723e-02 \\
  		& $10^3$ & 2.9233e-09 & 1.0386 & 1587.1 & 2.9854e-03 & 2.5363e-02 & 8.6687e-02 \\
\noalign{\smallskip}\hline\noalign{\smallskip}
3  		& $10^0$ & 2.1966e-08 & 0.9785 & 1591.7 & 1.0941e-03 & 2.3011e-02 &8.7283e-02  \\
 		& $10^1$ & 1.3851e-08 & 0.9963 & 1666.6 & 2.4658e-03 & 2.6271e-02 & 9.5628e-02 \\
  		& $10^2$ & 2.9369e-09 & 0.9434 & 1710.8 & 1.7993e-03 & 2.6604e-02 & 9.5979e-02 \\
  		& $10^3$ & 2.9860e-09 & 0.9914 & 1613.8 & 1.7759e-03 & 2.6611e-02 & 9.5978e-02 \\
\noalign{\smallskip}\hline
\end{tabular}
\end{table}

Given the precise approximation quality, we then use the RB method to obtain parameter estimates from noisy data for 100 different noise vectors.
Table \ref{table:muapprox:noise}(b) provides the minimum, mean and maximum distance between the RB parameter estimates $\mymui$ obtained from noisy and noise-free data.
We observe that the noise influenced the parameter estimate independently of $\lambda$ and $i$.
Further analysis showed that there was more deviation in the $(\mymuiRB)_1$-argument than in $(\mymuiRB)_2$, and that the mean over the biased parameter estimates converged against the noise-free $\mymuiRB$ when increasing the number of samples.


\section{Conclusion}
In this paper, we have proposed and analyzed a data-weak variant of the 3D-VAR method for parametrized PDEs. Through a data-informed perturbation of the model, the method generates an intermediate state between the best-knowledge model solution and the observation in the measurement space. We have reformulated the 3D-VAR method as a saddle-point problem and performed a stability analysis to reveal the relationship between the 3D-VAR method, the model, and the measurement space. In particular, we showed a necessary and sufficient condition on the design of the measurement space that results in improved stability and the mitigation of noise amplification. This condition can be used in the modelling process and in the experimental design for the selection of suitable model modifications and measurement functionals.

For an efficient computation in a parametrized real-time or many-query setting, a certified RB method was introduced. We developed {\it a posteriori} error bounds for the computationally efficient approximation of the error in the model modification, state, adjoint, and observable misfit between the truth and the RB solution. We proposed a greedy-OMP algorithm for choosing the measurement space and a construction of the RB spaces which does not require the measurement data to be known {\it a priori}. We presented numerical results for parameter and state estimation for a steady heat conduction problem with uncertain parameters and an unknown Neumann boundary condition. The numerical results confirm the validity of our approach as well as the theoretical findings.


\bibliographystyle{spmpsci}       
\bibliography{literature}   

\end{document}

%% file: thermalblock1.pdf_tex
\begingroup%
  \makeatletter%
  \providecommand\color[2][]{%
    \errmessage{(Inkscape) Color is used for the text in Inkscape, but the package 'color.sty' is not loaded}%
    \renewcommand\color[2][]{}%
  }%
  \providecommand\transparent[1]{%
    \errmessage{(Inkscape) Transparency is used (non-zero) for the text in Inkscape, but the package 'transparent.sty' is not loaded}%
    \renewcommand\transparent[1]{}%
  }%
  \providecommand\rotatebox[2]{#2}%
  \ifx\svgwidth\undefined%
    \setlength{\unitlength}{324.15243314bp}%
    \ifx\svgscale\undefined%
      \relax%
    \else%
      \setlength{\unitlength}{\unitlength * \real{\svgscale}}%
    \fi%
  \else%
    \setlength{\unitlength}{\svgwidth}%
  \fi%
  \global\let\svgwidth\undefined%
  \global\let\svgscale\undefined%
  \makeatother%
  \begin{picture}(1,1.1354019)%
    \put(0,0){\includegraphics[width=\unitlength,page=1]{thermalblock1.pdf}}%
    \put(0.196894,0.20829274){\color[rgb]{0,0,0}\makebox(0,0)[lb]{\smash{$x_1$}}}%
    \put(0.09329301,0.3350071){\color[rgb]{0,0,0}\makebox(0,0)[lb]{\smash{$x_2$}}}%
    \put(-3.80722179,-4.77186588){\color[rgb]{0,0,0}\makebox(0,0)[lt]{\begin{minipage}{1.31096999\unitlength}\raggedright \end{minipage}}}%
    \put(0.3613797,0.26746286){\color[rgb]{0,0,0}\makebox(0,0)[lb]{\smash{$\Omega_0,~\mu_0 = 1$}}}%
    \put(0.30042123,0.01936488){\color[rgb]{0,0,0}\makebox(0,0)[lb]{\smash{$\Gamma_{in},~\nabla y \cdot n = u$}}}%
    \put(-0.12909385,0.663893){\color[rgb]{0,0,0}\makebox(0,0)[lb]{\smash{$\Gamma_N$}}}%
    \put(1.03447347,0.663893){\color[rgb]{0,0,0}\makebox(0,0)[lb]{\smash{$\Gamma_N$}}}%
    \put(0.35595077,1.08364316){\color[rgb]{0,0,0}\makebox(0,0)[lb]{\smash{$\Gamma_D,~y = 0$}}}%
    \put(1.03447347,0.57507136){\color[rgb]{0,0,0}\makebox(0,0)[lb]{\smash{$\nabla y \cdot n = 0$}}}%
    \put(-0.34665705,0.57507136){\color[rgb]{0,0,0}\makebox(0,0)[lb]{\smash{$\nabla y \cdot n = 0$}}}%
    \put(0.40302685,0.59601331){\color[rgb]{0,0,0}\makebox(0,0)[lb]{\smash{$\Omega_1,~\mu_1$}}}%
    \put(0.40302685,0.91530764){\color[rgb]{0,0,0}\makebox(0,0)[lb]{\smash{$\Omega_2,~\mu_2$}}}%
    \put(0,0){\includegraphics[width=\unitlength,page=2]{thermalblock1.pdf}}%
  \end{picture}%
\endgroup%